\documentclass[a4paper,reqno,11pt]{amsart}

\usepackage{amssymb,amsmath,amsfonts,amsthm,amscd,amstext,mathtools}
\usepackage[english]{babel}
\usepackage[latin9]{inputenc}
\usepackage[version=3]{mhchem}
\usepackage{subfig}
\usepackage{cite}
\usepackage{graphicx}
\usepackage{perpage} 
\usepackage{url}
\usepackage{color}

\usepackage[a4paper,scale={0.72,0.74},marginratio={1:1},footskip=7mm,headsep=10mm]{geometry}

\usepackage[pdftex]{hyperref}

\numberwithin{equation}{section}

\newtheorem{theorem}{Theorem}[section]
\newtheorem{lemma}[theorem]{Lemma}
\newtheorem{proposition}[theorem]{Proposition}

\newtheorem{remark}[theorem]{Remark}

\newtheorem{step}{Step}

\renewcommand{\tilde}{\widetilde}          
\DeclareMathSymbol{\leqslant}{\mathalpha}{AMSa}{"36} 
\DeclareMathSymbol{\geqslant}{\mathalpha}{AMSa}{"3E} 
\DeclareMathSymbol{\eset}{\mathalpha}{AMSb}{"3F}     

\newcommand{\N}{\mathbb{N}}

\renewcommand{\epsilon}{\varepsilon}
\renewcommand{\theta}{\vartheta}
\renewcommand{\phi}{\varphi}

\newenvironment{myenumerate}{%
\renewcommand{\theenumi}{\arabic{enumi}}%
\renewcommand{\labelenumi}{{\rm(\theenumi)}}%
\begin{list}{\labelenumi}
        {%
        \setlength{\itemsep}{0.4em}%
        \setlength{\topsep}{0.5em}%
        \setlength\leftmargin{2.45em}%
        \setlength\labelwidth{2.05em}%
        \setlength{\labelsep}{0.4em}%
        \usecounter{enumi}%
        }%
        }%
{\end{list}
}

{\end{myenumerate}}

\newenvironment{myitemize}{%
\begin{list}{$\bullet$}%
        {%
        \setlength{\itemsep}{0.4em}%
        \setlength{\topsep}{0.5em}%
        \setlength\leftmargin{2.45em}%
        \setlength\labelwidth{2.05em}%
        \setlength{\labelsep}{0.4em}%
        }%
        }%
{\end{list}}

\renewenvironment{itemize}{
\begin{myitemize}}%
{\end{myitemize}}

\MakePerPage[2]{footnote} 

\author{Gia Bao Nguyen}

\address{Laboratoire de Math\'ematiques Jean Leray UMR 6629,
Universit\'e de Nantes, 2 Rue de la Houssini\`ere,
BP 92208, F-44322 Nantes Cedex 03, France}

\email{gia-bao.nguyen@univ-nantes.fr}

\author{Nicolas P\'etr\'elis}

\address{Laboratoire de Math\'ematiques Jean Leray UMR 6629,
Universit\'e de Nantes, 2 Rue de la Houssini\`ere,
BP 92208, F-44322 Nantes Cedex 03, France}

\email{nicolas.petrelis@univ-nantes.fr}



\thanks{{\it Acknowledgements.} We thank Philippe Carmona for fruitful discussions.}

\date{\today}

\title[Polymer Collapse]
{A variational formula for the free energy of the partially directed polymer collapse.}

\date{\today}

\begin{document}

\begin{abstract}
Long linear polymers in dilute solutions are known to undergo a collapse transition from a random coil (expand itself) to a compact ball (fold itself up) when the temperature is lowered, or the solvent quality deteriorates. A natural model for this phenomenon is a $1+1$ dimensional self-interacting and partially directed self-avoiding walk. In this paper, we develop a new method to study the partition function of this model, from which we derive a variational formula for the free energy. This variational formula allows us to prove the existence of the collapse transition and to identify the critical temperature in a simple way. We also prove that the order of the collapse transition is $3/2$.
\end{abstract}
\maketitle

\section{Introduction}

\subsection{The model}\label{model}
The spatial configurations of the polymer of length $L$ ($L$ monomers) are modelled by the trajectories of a partially directed random walk on $\mathbb{Z}^2$. This random walk is self-avoiding and does not take any step in the negative x-direction. More precisely, we let $\vec{e_1}=(1,0),\vec{e_2}=(0,1)$ denote the canonical basis of $\mathbb{Z}^2$ and  we choose the set of allowed $L$-step paths as:
\begin{align}\label{defWL}
\nonumber\mathcal{W}_L=\{w=(w_i)_{i=0}^L\in(\mathbb{N}_0\times\mathbb{Z})^{L+1}:\,&w_0=0,\\
\nonumber&w_{i+1}-w_i\in\{\vec{e_1},\vec{e_2},-\vec{e_2}\}\;\forall 0\leq i<L-1,\\
\nonumber&w_i\neq w_j\;\forall 0\leq i<j\leq L,\\
&w_L-w_{L-1}=\vec{e_1}\}.
\end{align}

Note that the choice of $w$ ending with an horizontal step is made for convenience only. Let us introduce two different laws on $\mathcal{W}_L$, uniform and non-uniform, denoted by $\mathbf{P}^\mathsf{u}_L$ and $\mathbf{P}^\mathsf{nu}_L$, respectively.\\
\\
(1) The uniform model: all $L$-step paths have the same probability, i.e.,
\begin{equation}
\mathbf{P}^\mathsf{u}_L(w)=\frac{1}{|\mathcal{W}_L|},\quad w\in\mathcal{W}_L.
\end{equation}

\noindent (2) The non-uniform model: the $L$-step paths have the following law
\begin{itemize}
  \item At the origin or after an horizontal step: the walker must step north, south or east with equal probability $1/3$.
  \item After a vertical step north (respectively south): the walker must step north (respectively south) or east with probability $1/2$.
\end{itemize}
For later convenience, the law on $\mathcal{W}_L$ is denoted by $\mathbf{P}^\mathsf{m}_L$, where $\mathsf{m}\in\{\mathsf{u},\mathsf{nu}\}$.

Any non-consecutive vertices of the walk though adjacent on the lattice are called \textit{self-touching} (Fig.~\ref{fig:self}). To take into account the interactions between monomers, we assign an energetic reward $\beta\geq0$ to the polymer for each of its self-touching. Thus, we associate with every random walk trajectory $w=(w_i)_{i=0}^L\in\mathcal{W}_L$ the Hamiltonian 
\begin{equation}\label{eq:Hal}
H_{L,\beta}(w):=\beta\sum_{\substack{i,j=0\\i<j-1}}^L\mathbf{1}_{\{\lVert w_i-w_j\rVert=1\}},
\end{equation}
such that the partition function of the model can be written as
\begin{equation}\label{eq:dist}
Z^\mathsf{m}_{L,\beta}=\sum_{w\in\mathcal{W}_L}e^{H_{L,\beta}(w)}\,\mathbf{P}^\mathsf{m}_L(w),\quad \mathsf{m}\in\{\mathsf{u},\mathsf{nu}\}.
\end{equation}

\begin{figure}\centering
	\includegraphics[width=.55\textwidth]{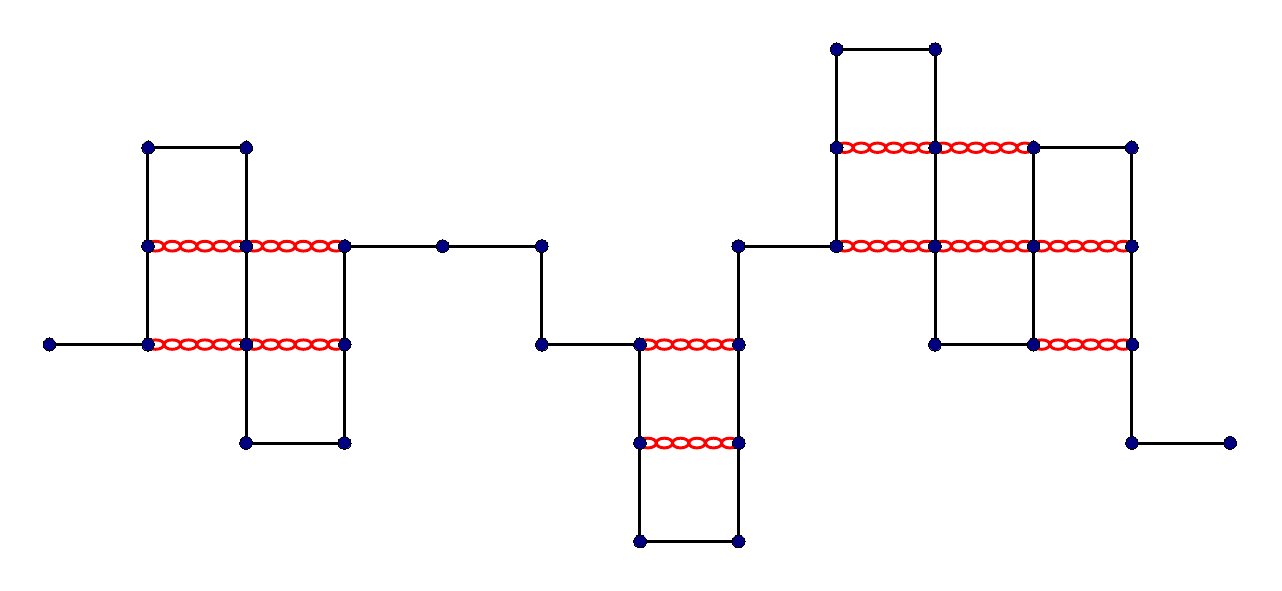}
	\caption{A partially directed walk with 12 self-touchings represented by light shaded bonds.}
	\label{fig:self}
\end{figure}

\subsection{Background}
The model presented in Section \ref{model} is often refered to in the physics litterature as the Interacting partially directed self-avoiding walk (IPDSAW). The IPDSAW was introduced in \cite{ZL68} as a directed version of the Interacting self-avoiding walk (ISAW) for which the set $\mathcal{W}_L$ is similar to what we defined in \eqref{defWL} except that steps to the west ($-\vec{e_1}$) are allowed as well. 

The ISAW allows for a better understanding of the geometric configuration adopted by an homopolymer dipped in a poor solvent. The monomers constituting the polymer try to exclude the solvent and therefore attract one another. Consequently, at low temperature, the polymer will fold on itself to form a \textit{compact ball}. The transition from an extended configuration to a compact ball is called \textit{collapse transition}. Both ISAW and IPDSAW are known to undergo a collapse transition at some critical temperature $T_c=1/\beta_c$. Detecting the phase transition requires to spot the temperature at which the free energy $f$ of the model loses its analyticity. In \cite{BGW92} or \cite{OPB92}, the method employed consists in providing an analytic expression of the generating function $G(z)=\sum_{L=1}^\infty Z_{L,\beta}^\mathsf{m}z^L$ whose radius of convergence $R$ satisfies $f=-\log R$. The idea behind the computation of the generating function is to rewrite $G(z)$ under the form $\sum_{r=0}^\infty g_r(z)$ where $g_r(z)$ is the contribution to $G(z)$ of those trajectories making exactly $r$ consecutive vertical steps at the beginning, regardless of the total length of the trajectory. By applying some smart paths concatenation, a recurrence relation is obtained between $g_{r-1},g_r$ and $g_{r+1}$ and then, after making the ansatz that $g_r$ can be expressed as an infinite sum, the recurrence relation allows for an exact computation of the terms in the infinite sum that provides $g_r$. For a detailed version of the computation, we can refer to \cite[p.~371--375]{CHP12}.
 
The same method has subsequently been applied to some variations of the IPDSAW, for instance in \cite{BDL09} where a force is applied at the right extremity of the polymer or in \cite{OPB92} where a continuous version of the model is studied.

One of the main difficulty arising from the computation of the generating function $G$ is that its analytic expression is very complicated and only gives an undirect access to the free energy. Our aim in this paper is to present a new method, that allows to work directly with the partition function of finite size. We will provide a variational formula for the free energy, from which the critical temperature can be computed easily. With the help of this variational formula, we will also give a rigorous proof of the fact that the collapse transition is of order $3/2$. Such a proof was lacking up to now.

\subsection{A new approach}
We partition the set $\mathcal{W}_L$ into $L$ subsets, each of them containing those trajectories that have the same number of horizontal steps. Via an algebraic manipulation of the Hamiltonian, it turns out that the contribution to the partition function $Z^\mathsf{m}_{L,\beta}$ in \eqref{eq:dist} of those trajectories making exactly $N$ horizontal steps can be expressed in a convenient manner. To be more specific, this contribution is proportional to a constant term (depending on $\beta$ only) at power $N$ times the probability that a symetric random walk, whose law $\mathbf{P}_{\beta}$ will be defined in \eqref{lawP}, satisfies some geometric constraints. Thus, we have
\begin{equation}\label{eq:partfuncp1}
Z^\mathsf{m}_{L,\beta}\sim\sum_{N=1}^{L}\big(\Gamma^\mathsf{m}(\beta)\big)^N\,\mathbf{P}_{\beta}\big(\mathcal{V}_{N+1,L-N}\big),
\end{equation}
where
\begin{equation*}
\mathcal{V}_{n,k}:=\big\{(V)_{i=0}^n:\,\textstyle\sum_{i=1}^n|V_i|=k,\,V_{n}=0\big\},
\end{equation*}
where $\beta\mapsto\Gamma^\mathsf{m}(\beta)$ is a continuous and decreasing bijection from $(0,\infty)$ to $(0,\infty)$ and where $(V_i)_{i\in \mathbb{N}}$ is a random walk with geometric increments. The formula in \eqref{eq:partfuncp1} will be made rigorous in Section \ref{na}, but the phase diagram of the model can already be read on this formula. In dependence of the value taken by $\Gamma^\mathsf{m}(\beta)$, we can indeed distinguish between the $3$ regimes displayed by the model:\\
\begin{itemize}
	\item $\Gamma^\mathsf{m}(\beta)>1$: \emph{the extended regime}. For $c\in(0,1)$, the quantities $\mathbf{P}_{\beta}\big(\mathcal{V}_{cL,L(1-c)}\big)$ are decaying exponentially fast when $L\to \infty$, at a rate which grows with $c$. Thus, the leading terms in \eqref{eq:partfuncp1} are those indexed by $N\sim\tilde{c}L$, where $\tilde{c}\in(0,1)$ is the result of an optimization. This phase is extended because those trajectories that are mainly contributing to the partition function have an horizontal length $N$ and a total length $L$ of the same order (Fig.~\ref{fig:phases}).\\
	\item $\Gamma^\mathsf{m}(\beta)=1$: \emph{the critical regime}. The leading terms in \eqref{eq:partfuncp1} are those indexed by $N$ of order $L^{2/3}$, because the quantity $\mathbf{P}_{\beta}\big(\mathcal{V}_{N+1,L-N}\big)$ reaches its maximum for such values of $N$.\\
	\item $\Gamma^\mathsf{m}(\beta)<1$: \emph{the collapsed regime}. For $c\in(0,\infty)$, the quantities $\mathbf{P}_{\beta}\big(\mathcal{V}_{c\sqrt{L},L}\big)$ are decaying like $e^{-t_c\sqrt{L}}$ where $t_c>0$ is decreasing in $c$. Thus, the leading terms in \eqref{eq:partfuncp1} are those indexed by $N\sim\hat{c}\,\sqrt{L}$, where $\hat{c}\in(0,\infty)$ is again the result of an optimization. This phase is collapsed because the trajectories that are mainly contributing to the partition function have an horizontal length $N$ much smaller than their total length $L$ (Fig.~\ref{fig:phases}).
\end{itemize}
\begin{figure}[ht]\centering
	\begin{tabular}{c c}
	\includegraphics[width=.3\textwidth]{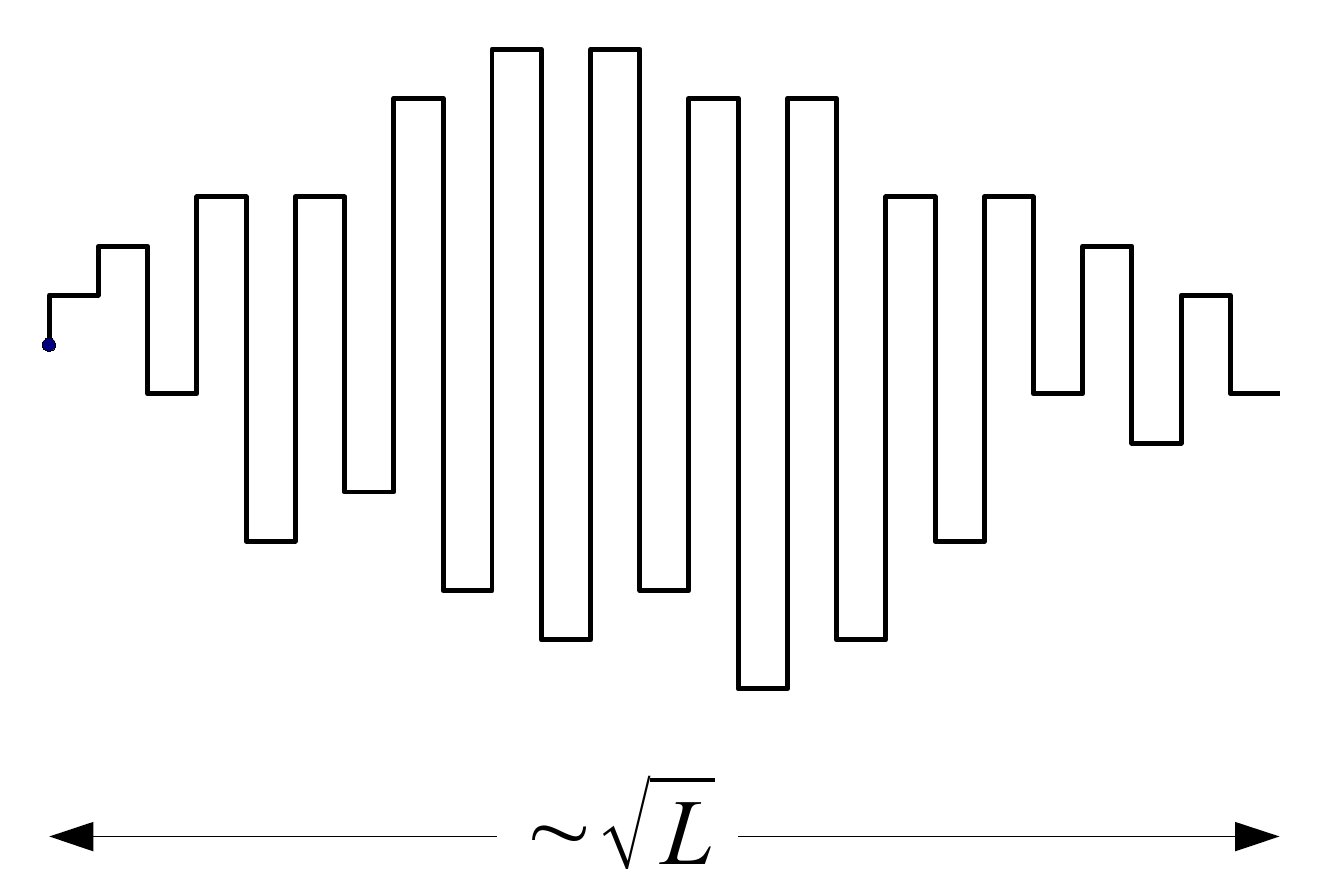} & \includegraphics[width=.5\textwidth]{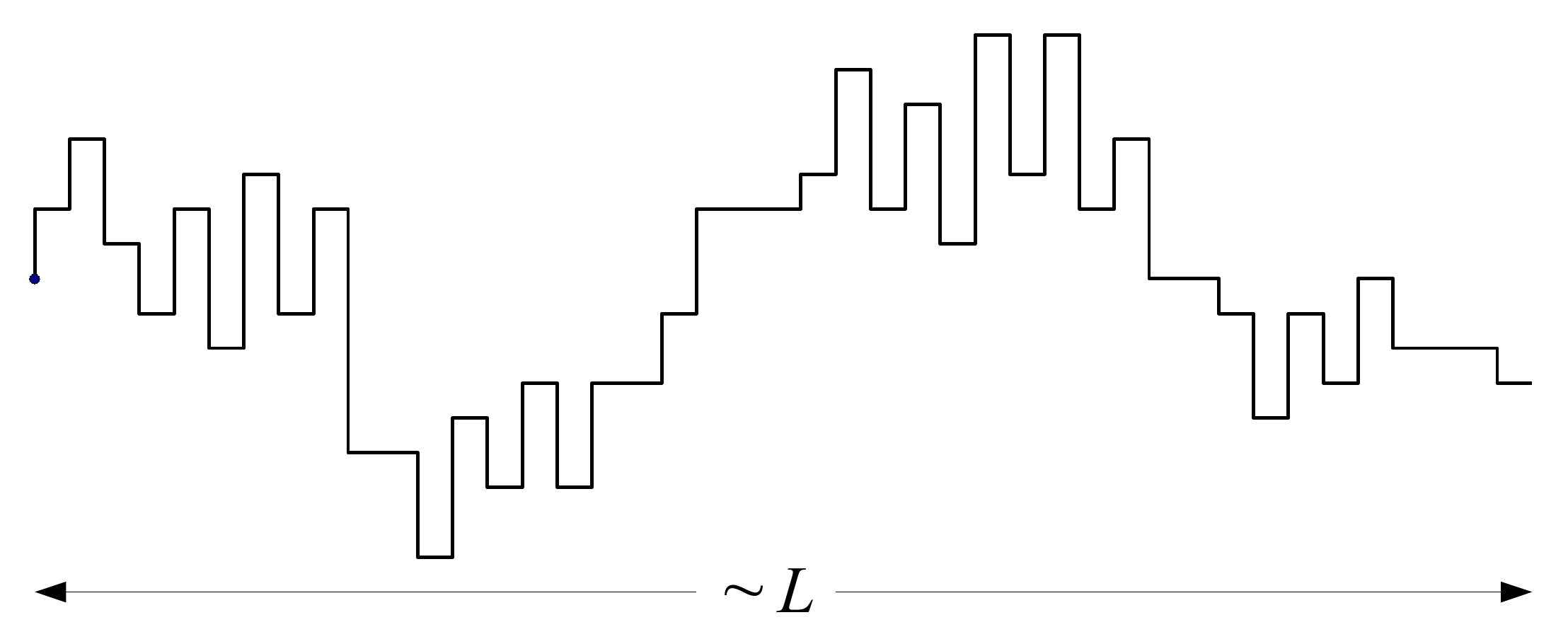}\\
	\small{Collapsed} & \small{Extended}\\
	\end{tabular}
	\caption{A typical path of both phases.}
	\label{fig:phases}
\end{figure}

\subsection{Main results}
For both models, i.e., $\mathsf{m}\in\{\mathsf{u},\mathsf{nu}\}$, the free energy per step $f^\mathsf{m}:(0,\infty)\to\mathbb{R}$ is defined as the limit
\begin{equation}\label{eq:fenergy}
f^\mathsf{m}(\beta):=\lim_{L\to\infty}f^\mathsf{m}_L(\beta),\quad\text{where}\quad f^\mathsf{m}_L(\beta):=\frac{1}{L}\log Z^\mathsf{m}_{L,\beta}.
\end{equation}
Note that $\{\log Z^\mathsf{m}_{L,\beta}\}_L$ is a subadditive sequence and since the number of self-touchings is smaller than the number of monomers, i.e. $H_{L,\beta}(w)\leq \beta L$, we immediately obtain the upper bound $Z^\mathsf{m}_{L,\beta}\leq e^{\beta L}$ for $\beta\in(0,\infty)$ and $\mathsf{m}\in\{\mathsf{u},\mathsf{nu}\}$. Then the limit in \eqref{eq:fenergy} exists and is finite
\begin{equation}
f^\mathsf{m}(\beta)=\lim_{L\to\infty}\frac{1}{L}\log Z^\mathsf{m}_{L,\beta}=\sup_{L\in\mathbb{N}}\frac{1}{L}\log Z^\mathsf{m}_{L,\beta}\leq\beta.
\end{equation}

If we shut down the self-interaction of the polymer, that is if we take $\beta=0$, then the density of self-touching performed by a typical $L$-step random walk trajectory belongs to $(0,1)$, the horizontal extension of this trajectory is of order $L$ and its vertical displacement of order $\sqrt{L}$. When $\beta$ becomes strictly positive, in turn, the geometric conformation adopted by the random walk is the result of an "energy-entropy" competition which can be understood as follows. To increase its self-touching density, the copolymer must both, reduce its number of horizontal steps and constrain its consecutive sequences of vertical steps to take opposite directions. However, these two geometric constraints have an entropic cost such that the free energy is the result of an optimization between the energetic gain and the entropic cost induced by a raise of the self-touching density. When $\beta$ becomes large, the system enters its collapsed phase which corresponds to a saturation of the self-touchings made by the polymer. In other words, the collapsed configurations have a self-touching density equal to $1$, which necessarily entails that the number of horizontal steps made by such configurations is of order $o(L)$ and that most pairs of consecutive vertical stretches are of opposite directions. These geometric restrictions are associated with a collapsed entropy $\kappa_\mathsf{m}$ for $\mathsf{m}\in\{\mathsf{u},\mathsf{nu}\}$ such that the free energy takes the form $\beta+\kappa_\mathsf{m}$. In Lemma \ref{le1} below, we display the value of this collapsed entropy.

\begin{lemma}\label{le1}
For $\beta>0$, $\mathsf{m}\in\{\mathsf{u},\mathsf{nu}\}$
\begin{equation}
f^\mathsf{m}(\beta)\geq\phi^\mathsf{m}_\beta, 
\end{equation}
where $\phi^\mathsf{u}_\beta=\beta-\log{(1+\sqrt{2})}$ and $\phi^\mathsf{nu}_\beta=\beta-\log2$.
\end{lemma}
\begin{proof}
We pick $L$ such that $\sqrt{L}\in\mathbb{N}$ and restrict the sum giving $Z^\mathsf{m}_{L,\beta}$ to a single $L$-step trajectory $\widetilde{w}$ which starts with $\sqrt{L}-1$ steps north then makes one step east, then $\sqrt{L}-1$ steps south, then one step east, then $\sqrt{L}-1$ steps north and so on... This trajectory makes $\sqrt{L}$ horizontal steps, separating $\sqrt{L}$ vertical stretches of length $\sqrt{L}-1$ each. Since any two consecutive vertical stretches of $\widetilde{w}$ have opposite direction, its Hamiltonian is given by $\beta(\sqrt{L}-1)^2\geq\beta L-2\beta\sqrt{L}$. Moreover, $\mathbf{P}^\mathsf{u}_L(\widetilde{w})=1/|\mathcal{W}_L|$ and $\mathbf{P}^\mathsf{nu}_L(\widetilde{w})=(2/3)^{\sqrt{L}}(1/2)^L$, therefore
\begin{equation}\label{ee}
Z^\mathsf{u}_{L,\beta}\geq\frac{e^{\beta(L-2\sqrt{L})}}{|\mathcal{W}_L|}\quad\text{and}\quad
Z^\mathsf{nu}_{L,\beta}\geq\left(\frac{e^\beta}{2}\right)^L\left(\frac{2}{3e^{2\beta}}\right)^{\sqrt{L}}.
\end{equation}
Since $\lim_{L\to\infty}L^{-1}\log|\mathcal{W}_L|=\log{(1+\sqrt2)}$ (see \cite[p.~5]{MBM10}), it remains to take $\tfrac{1}{L}\log$ in each term of the two inequalities in \eqref{ee} and to let $L\to\infty$ to complete the proof of the Lemma. 
\end{proof}

Clearly, all what Lemma \ref{le1} is saying is that $\kappa_\mathsf{u}\geq -\log2$ and that $\kappa_\mathsf{nu}\geq -\log(1+\sqrt{2})$. However, we will see below that these two inequalities are in fact equalities.

Let us define the \textit{excess free energy} $\tilde{f}^\mathsf{m}(\beta):=f^\mathsf{m}(\beta)-\phi^\mathsf{m}_\beta$, which by Lemma \ref{le1} above is always non negative. The lower bound in Lemma \ref{le1} allows us to partition $[0,\infty)$ into a collapsed phase denoted by $\mathcal{C}$ and an extended phase denoted by $\mathcal{E}$, i.e,
\begin{equation}
\mathcal{C}:=\{\beta:\tilde{f}^\mathsf{m}(\beta)=0\}
\end{equation}
and
\begin{equation}
\mathcal{E}:=\{\beta:\tilde{f}^\mathsf{m}(\beta)>0\}.
\end{equation}
Since $\tilde{f}^\mathsf{m}(\beta)$ is convex, non negative and bounded above, proving that there exists $\beta^\mathsf{m}_0\in [0,\infty)$ such that $\tilde{f}^\mathsf{m}(\beta^\mathsf{m}_0)=0$ will be sufficient to claim that $\tilde{f}^\mathsf{m}(\beta)=0$ for $\beta\geq\beta^\mathsf{m}_0$. Then, the critical point will be defined as 
\begin{equation}
\beta^\mathsf{m}_c:=\inf\{\beta\ge0:\tilde{f}^\mathsf{m}(\beta)=0\},
\end{equation}
and the sets $\mathcal{C}$ and $\mathcal{E}$ will become $\mathcal{C}=\{\beta:\beta\geq\beta^\mathsf{m}_c\}$ and $\mathcal{E}=\{\beta:\beta<\beta^\mathsf{m}_c\}$.

\subsubsection{Variational formula} In Theorem \ref{Thm1} below we provide a variational formula for the excess free energy of the model. We need to settle some of the ingredients appearing in the formula. We let $V:=(V_n)_{n\in\mathbb{N}}$ be a symmetric random walk on $\mathbb{Z}$, whose increments are independent and follow a geometric distribution, i.e. $V_0=0$, $V_n=\sum_{i=1}^n v_i$ for $n\in\mathbb{N}$ and $v:=(v_i)_{i\in\mathbb{N}}$ is an i.i.d sequence under the law $\mathbf{P}_{\beta}$, with distribution 
\begin{equation}\label{lawP}
\mathbf{P}_{\beta}(v_1=k)=\tfrac{e^{-\frac{\beta}{2}|k|}}{c_{\beta}}\quad\forall k\in\mathbb{Z}\quad\text{with}\quad c_{\beta}:=\tfrac{1+e^{-\beta/2}}{1-e^{-\beta/2}}.
\end{equation}
For each $\alpha\in[0,\infty)$, we set
\begin{equation}\label{eq:glimit1}
g_\beta(\alpha):=\lim_{N\to\infty}\tfrac{1}{N}\log\mathbf{P}_{\beta}\Big(\sum_{i=1}^N|V_i|\leq\alpha N,\,V_N=0\Big).
\end{equation}
We will prove in Section \ref{gbeta} below that the limit in \eqref{eq:glimit1} exists and that $\alpha\mapsto g_\beta(\alpha)$ is negative, concave, increasing on $[0,\infty)$ and converges to $0$ as $\alpha\to \infty$. Finally, we define the function $\Gamma^\mathsf{m}:(0,\infty)\to(0,\infty)$, for $\mathsf{m}\in\{\mathsf{u},\mathsf{nu}\}$, as
\begin{equation}\label{sqq}
	\begin{dcases*}
	\Gamma^\mathsf{u}(\beta)=\tfrac{c_\beta}{e^\beta},\\
  \Gamma^\mathsf{nu}(\beta)=\tfrac{2c_\beta}{3e^\beta}.
  \end{dcases*}
\end{equation}

\begin{theorem}[Variational formula]\label{Thm1}
For $\mathsf{m}\in\{\mathsf{u},\mathsf{nu}\}$, the excess free energy $\tilde{f}^\mathsf{m}(\beta)$ is given by
\begin{equation}\label{eq:main}
\tilde{f}^\mathsf{m}(\beta)=\sup_{\alpha\in[0,1]}\left[\alpha\log\left(\Gamma^\mathsf{m}(\beta)\right)+\alpha\,g_\beta\left(\tfrac{1-\alpha}{\alpha}\right)\right].
\end{equation}
\end{theorem}
A consequence of Theorem \ref{Thm1} is that there exists a critical point $\beta^\mathsf{m}_c>0$ at which the polymer undergoes the collapse transition (Fig.~\ref{fig:crit}) and that $\beta_c^\mathsf{m}$ can be computed explicitly. 

\begin{theorem}[Critical point]\label{Thm2}
For $\mathsf{m}\in\{\mathsf{u},\mathsf{nu}\}$, there exists a $\beta_c^\mathsf{m}\in(0,\infty)$ such that 
\begin{equation}
\tilde{f}^\mathsf{m}(\beta)\begin{dcases*}
	=0, & if $\beta\geq\beta^\mathsf{m}_c$,\\
  >0, & if $\beta<\beta^\mathsf{m}_c$,
  \end{dcases*}
\end{equation}
and $\beta_c^\mathsf{m}$ is the unique positive solution of the equation $\Gamma^\mathsf{m}(\beta)=1$.
\end{theorem}
By recalling \eqref{lawP} and \eqref{sqq}, we observe that the equation $\Gamma^\mathsf{nu}(\beta)=1$ is equivalent to the equation $3x^3-3x^2-2x-2=0$ where $x=e^{\beta/2}$. Moreover, the cubic polynomial $3x^3-3x^2-2x-2$ has a unique positive zero $x_c$, so that $\beta_c^\mathsf{nu}=2\log x_c$. This value of $\beta_c^\mathsf{nu}$ corresponds to the value provided in \cite{BGW92}.
\begin{figure}[ht]\centering
	\includegraphics[width=.5\textwidth]{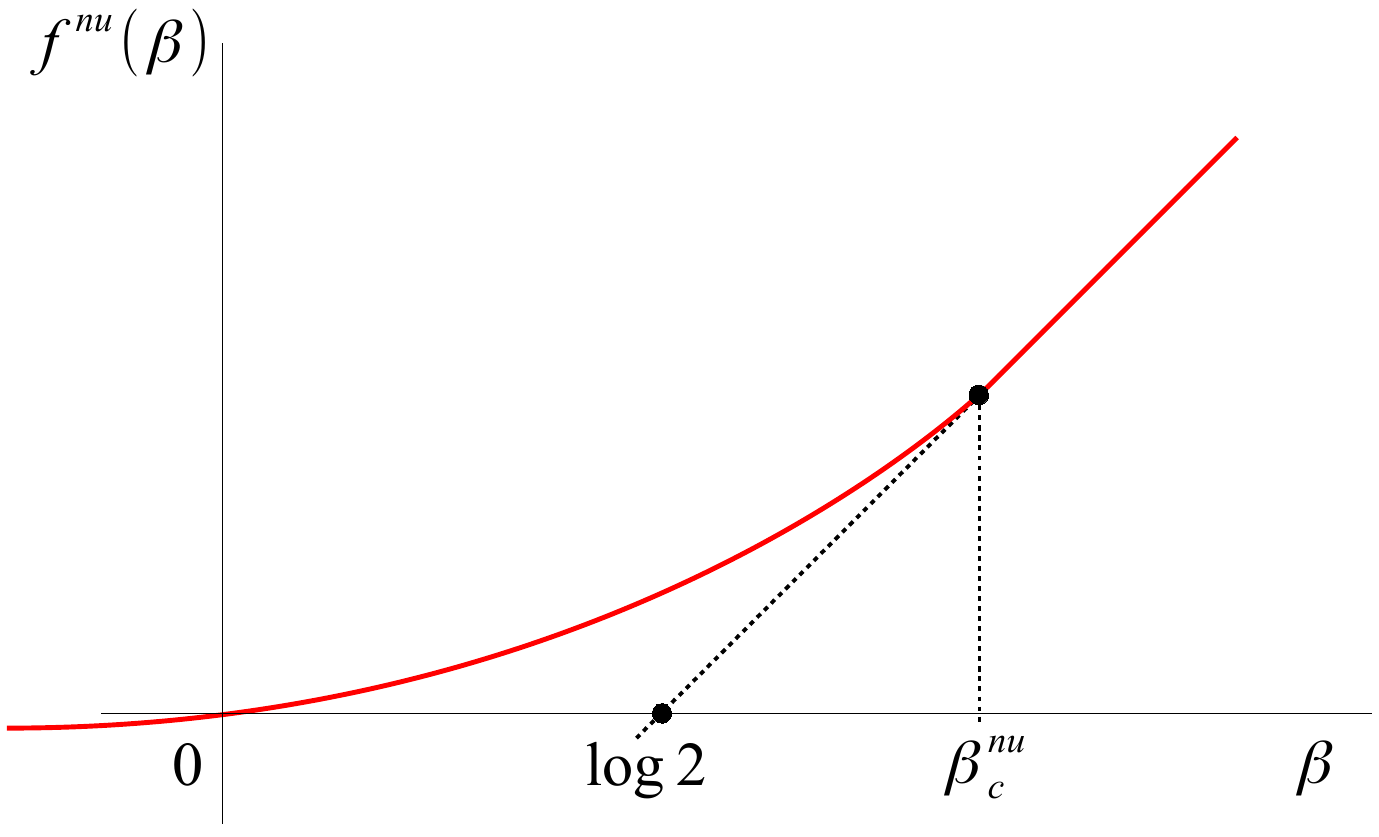}
	\caption{Phase diagram in the the non-uniform case ($\beta_c^\mathsf{nu}\approx1$).}
	\label{fig:crit}
\end{figure}

\begin{theorem}[Order of the phase transition]\label{Thm3}
The phase transition is of order $3/2$. That is, there exist two constants $c_1,c_2>0$ such that for $\epsilon$ small enough
\begin{equation}
c_1\,\epsilon^{3/2}\leq\tilde{f}^\mathsf{m}(\beta_c^\mathsf{m}-\epsilon)\leq c_2\,\epsilon^{3/2}.
\end{equation}
\end{theorem}

\section{A new representation of the partition function}\label{na}

With Proposition \ref{tt} below, we give a rigorous statement of the formula \eqref{eq:partfuncp1} which is the corner stone of our paper.

We need to settle a few notations before stating Proposition \ref{tt}. For that, we recall \eqref{lawP} and we let $V=(V_i)_{i\geq 0}$ be a random walk of law $\mathbf{P}_{\beta}$. We let
\begin{equation}\label{defA}
\textstyle A_n:=\sum_{i=1}^{n}|V_i|
\end{equation}
represent the area enclosed between the random walk $V$ and the horizontal line $y=0$. We let also $\mathcal{V}_{n,k}$ be the subset containing those trajectories that return to the origin after $n$ steps and satisfy $A_n=k$, that is
\begin{equation}\label{defVV}
\mathcal{V}_{n,k}:=\{(V)_{i=0}^n:\,V_n=0,\,A_n=k\}.
\end{equation}
Finally, we recall the definition of $\Gamma^\mathsf{m}(\beta)$ in \eqref{sqq} and we set
\begin{equation}\label{sqq1}
	\begin{dcases*}
	\Phi^\mathsf{u}_{L,\beta}=e^{\beta L}/|\mathcal{W}_L|,\\
	\Phi^\mathsf{nu}_{L,\beta}=(e^\beta/2)^L.
  \end{dcases*}
\end{equation}

\begin{proposition}\label{tt}
For $\beta>0$, $L\in\mathbb{N}$, $\mathsf{m}\in\{\mathsf{u},\mathsf{nu}\}$, we have
\begin{equation}\label{eq:partfunc}
Z^\mathsf{m}_{L,\beta}=c_\beta\,\Phi^\mathsf{m}_{L,\beta}\sum_{N=1}^{L}\left(\Gamma^\mathsf{m}(\beta)\right)^N\mathbf{P}_{\beta}(\mathcal{V}_{N+1,L-N}).
\end{equation}
\end{proposition}
\begin{proof}
We will display the details of the proof in the non-uniform case only. The uniform case is indeed easier to handle because the probability associated with each trajectory in $\mathcal{W}_L$ is constant. The walk can be decomposed into $N$ stretches $\gamma_1,\ldots,\gamma_N$, each of them consisting of a vertical part of length $l\in\mathbb{Z}$ and of one horizontal step. Thus, with each configuration $w\in\mathcal{W}_L$, we associate the sequence $l:=(l_1,\ldots,l_N)\in\mathbb{Z}^N$ such that $N$ is the number of vertical stretches made by $w$ and $l_i$ corresponds to the vertical length of the $i^{th}$ stretch for $i\in\{1,\dots,N\}$ (Fig.~\ref{fig:stretches}). At this stage, we have a one-to-one correspondence between $\mathcal{W}_L$ and $\Omega_L:=\bigcup_{N=1}^L\mathcal{L}_{N,L}$, where $\mathcal{L}_{N,L}$ is the set of all possible configurations consisting of $N$ vertical stretches that have a total length $L$, that is
\begin{equation}
\textstyle\mathcal{L}_{N,L}=\Bigl\{l\in\mathbb{Z}^N:\sum_{n=1}^N|l_n|+N=L\Bigr\}.
\end{equation}
\begin{figure}[hb]\centering
	\includegraphics[width=.32\textwidth]{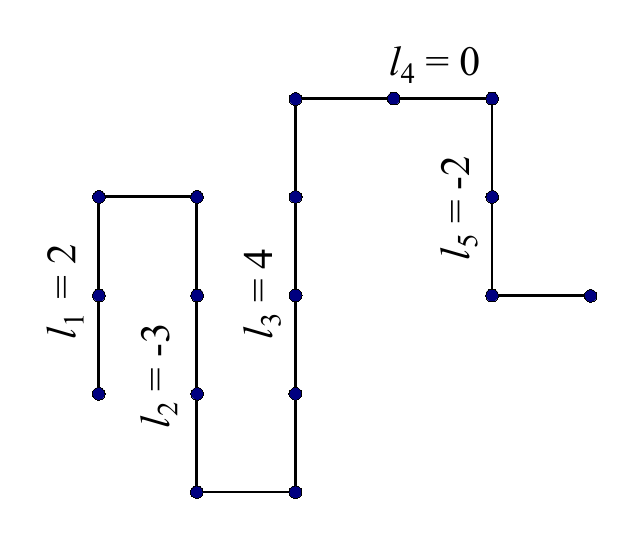}
	\caption{An example of decomposition of $w\in\mathcal{W}_{16}$ into 5 stretches.}
	\label{fig:stretches}
\end{figure}\\
By recalling the definition of $\mathbf{P}^\mathsf{nu}_L$ in Section \ref{model}, we note that the function $l\mapsto\mathbf{P}^\mathsf{nu}_L(l)$ is constant equal to $(2/3)^N (1/2)^L$ on each subset $\mathcal{L}_{N,L}$. Moreover, the Hamiltonian associated with a stretch configuration $l:=(l_1,\ldots,l_N)$ is given by
\begin{equation}
\textstyle H_{L,\beta}(l_1,\ldots,l_N)=\beta\sum_{n=1}^{N-1}(l_n\;\tilde{\wedge}\;l_{n+1})
\end{equation}
where
\begin{equation}
x\;\tilde{\wedge}\;y=\begin{dcases*}
	|x|\wedge|y| & if $xy<0$,\\
  0 & otherwise.
  \end{dcases*}
\end{equation}
The one-to-one correspondance between $\Omega_L$ and $\mathcal{W}_L$ allows us to rewrite the partition function in terms of the stretches, i.e.,
\begin{equation}\label{pf}
Z^\mathsf{nu}_{L,\beta}=\sum_{N=1}^{L}\sum_{l\in\mathcal{L}_{N,L}}\left(\tfrac{1}{3}\right)^N\left(\tfrac{1}{2}\right)^{L-N}e^{\beta\sum_{i=1}^{N-1}(l_i\;\tilde{\wedge}\;l_{i+1})}.
\end{equation}
At this stage, it is useful to remark that the $\tilde{\wedge}$ operator can be written as
\begin{equation}
x\;\tilde{\wedge}\;y=\left(|x|+|y|-|x+y|\right)/2,\quad\forall x,y\in\mathbb{Z}.
\end{equation}
Hence, for $\beta>0$ and $L\in\mathbb{N}$, the partition function in \eqref{pf} becomes
\begin{align}\label{ls}
\nonumber Z^\mathsf{nu}_{L,\beta}
&=\sum_{N=1}^{L}\left(\tfrac{1}{3}\right)^N\left(\tfrac{1}{2}\right)^{L-N}\sum_{\substack{l\in\mathcal{L}_{N,L}\\l_0=l_{N+1}=0}}\exp{\Bigl(\beta\sum_{n=1}^N{|l_n|}-\tfrac{\beta}{2}\sum_{n=0}^N{|l_n+l_{n+1}|}\Bigr)}\\
&=\left(\tfrac{e^\beta}{2}\right)^L\sum_{N=1}^{L}c_\beta\left(\tfrac{2c_\beta}{3e^\beta}\right)^N\sum_{\substack{l\in\mathcal{L}_{N,L}
\\l_0=l_{N+1}=0}}\prod_{i=0}^{N}\frac{\exp{\Bigl(-\tfrac{\beta}{2}|l_i+l_{i+1}|\Bigr)}}{c_\beta},
\end{align}
where $c_\beta$ was defined in \eqref{lawP}.
By rewriting the last sum in \eqref{ls} in terms of $v_n:=(-1)^{n-1}(l_{n-1}+l_n)$, $n=1,\ldots,N+1$, we see that this sum is equal to the probability that the random walk $(V_n)_{n\in\mathbb{N}}$ belongs to $\mathcal{V}_{N+1,L-N}$ (Fig.~\ref{fig:rwv}). Therefore
\begin{equation}
Z^\mathsf{nu}_{L,\beta}=c_\beta\left(\tfrac{e^\beta}{2}\right)^L\sum_{N=1}^{L}\left(\tfrac{2c_\beta}{3e^\beta}\right)^N\mathbf{P}_{\beta}(\mathcal{V}_{N+1,L-N}).
\end{equation}
\begin{figure}[ht]\centering
	\includegraphics[width=.8\textwidth]{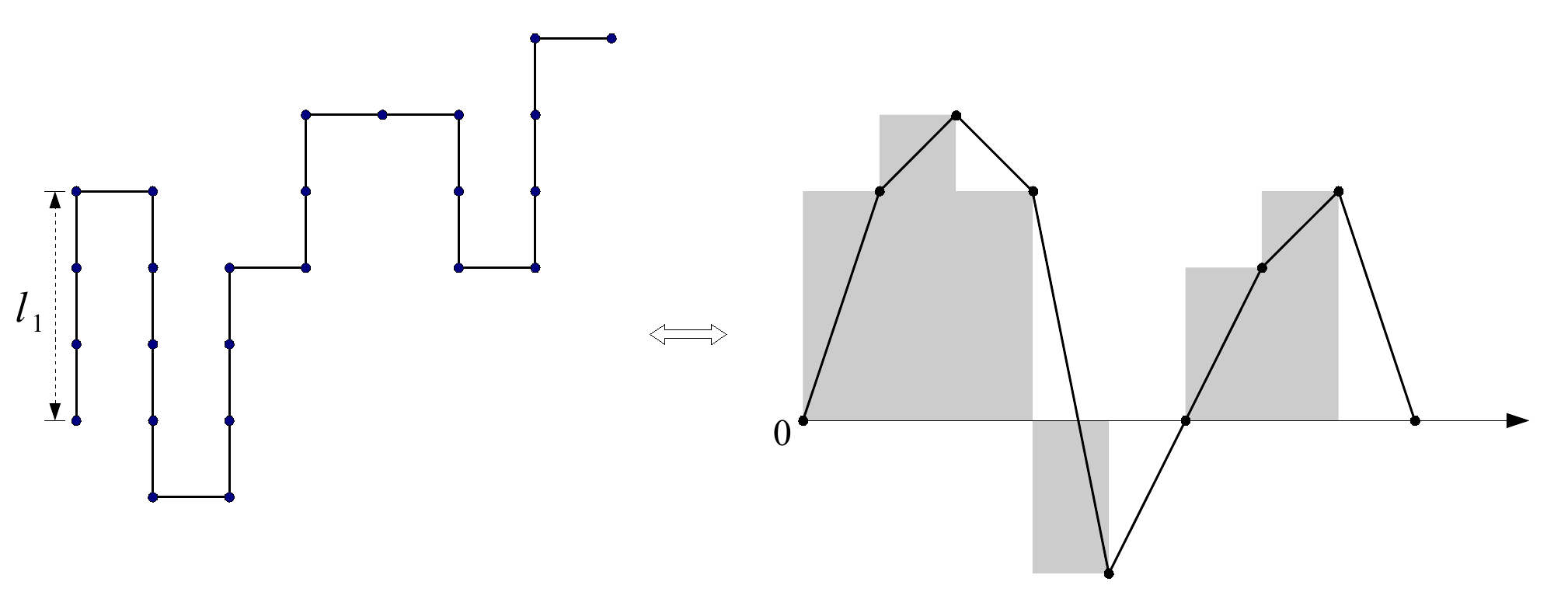}
  \caption{An example of transformation of $w\in \mathcal{W}_{24}$ into $(V_n)_{n=0}^8\in \mathcal{V}_{8,17}$. 
The 24-step trajectory $w$ on the left has 7 stretches: $l_1=3,l_2=-4,l_3=3,l_4=2,l_5=0,l_6=-2$ and $l_7=3$. The correspondent increments of $(V_n)_{n=0}^8$ are: $v_1=3,v_2=1,v_3=-1,v_4=-5,v_5=2,v_6=2$ and $v_7=1$.}
  \label{fig:rwv}
\end{figure}
\end{proof}

\section{Construction and asymptotics of \texorpdfstring{$g_\beta$}{g}}\label{gbeta}

In Section \ref{creg}, we construct rigorously the entropic function $g_\beta$ and we study its regularity and monotonicity. In section \ref{assym}, we focus on the asymptotic behavior of $g_\beta(\alpha)$ when $\alpha\to\infty$.

\subsection{Construction and regularity of \texorpdfstring{$g_\beta$}{g}}\label{creg}
We now define the function $g_\beta$ in a sligthly different way from \eqref{eq:glimit1}, but we will see at the end of section \ref{creg} that the two definitions are equivalent. Recall \eqref{defVV} and for each $\alpha\in\mathbb{Q}^+:=\mathbb{Q}\cap[0,\infty)$, let 
\begin{equation}\label{nalpha}
N_\alpha:=\{n\in\mathbb{N}\cap[2,\infty):\,\alpha n\in\mathbb{N}\}.
\end{equation}
Note that $\mathbf{P}_{\beta}(\mathcal{V}_{N,\alpha K})>0$ for all $N\in\mathbb{N}\cap[2,\infty)$ and $K\in N_\alpha$. Let $g_\beta:\mathbb{Q}^+\to\mathbb{R}$ be defined as 
\begin{equation}\label{eq:glimit}
g_\beta(\alpha):=\lim_{\substack{N\in N_\alpha\\N\to\infty}}g_{N,\beta}(\alpha),\quad\text{where}\quad g_{N,\beta}(\alpha):=\frac{1}{N}\log\mathbf{P}_{\beta}(\mathcal{V}_{N,\alpha N}).
\end{equation}
To study the properties of $g_\beta$, we will use that 
\begin{equation}\label{propV1}
\mathbf{P}_{\beta}(\mathcal{V}_{N_1+N_2,K_1+K_2})\geq\mathbf{P}_{\beta}(\mathcal{V}_{N_1,K_1})\mathbf{P}_{\beta}(\mathcal{V}_{N_2,K_2}), \quad\text{for}\; N_1,N_2,K_1,K_2\in\mathbb{N}.
\end{equation}
To prove \eqref{propV1}, we simply restrict the set $\mathcal{V}_{N_1+N_2,K_1+K_2}$ to those trajectories that return to origin at time $N_1$ and satisfy $A_{N_1}=K_1$. Then, by using the Markov property of the pair process $\{V_n,A_n\}_{n\in\mathbb{N}}$, we obtain the result.

\begin{lemma}\label{funcg}
(i) $g_\beta(\alpha)$ exists and is finite, nonpositive for all $\alpha\in\mathbb{Q}^+$. In particular, $g_\beta(0)=-\log c_\beta$.\\
(ii) $\alpha\mapsto g_\beta(\alpha)$ is continuous, concave, nondecreasing on $\mathbb{Q}^+$ and tends to $0$ as $\alpha\to \infty$.
\end{lemma}
\begin{remark}
The continuity and the concavity of $g_\beta$ guarantee that it can be extended to a continuous function on $\mathbb{R}^+$.
\end{remark}
\begin{proof}
(i) Because of \eqref{propV1}, for $N_1,N_2\in N_\alpha$, we have
\begin{equation}
\mathbf{P}_{\beta}(\mathcal{V}_{N_1+N_2,\alpha(N_1+N_2)})\geq\mathbf{P}_{\beta}(\mathcal{V}_{N_1,\alpha N_1})\, \mathbf{P}_{\beta}(\mathcal{V}_{N_2,\alpha N_2}).
\end{equation}
Thus, $\{\log\mathbf{P}_{\beta}(\mathcal{V}_{N,\alpha N})\}_{N\in N_\alpha}$ is a subadditive sequence and since $0<\mathbf{P}_{\beta}(\mathcal{V}_{N,\alpha N})\leq1$ for $N\in N_\alpha$, the limit in \eqref{eq:glimit} exists, is finite and satisfies
\begin{equation}\label{eq:supadd}
g_\beta(\alpha)=\sup_{N\in N_\alpha}\frac{1}{N}\log\mathbf{P}_{\beta}(\mathcal{V}_{N,\alpha N})\leq0.
\end{equation}
We recall that $\mathcal{V}_{n,0}=\{(V)_{i=0}^n:\,V_n=0,\,A_n=0\}$, so that
\begin{equation}
\mathbf{P}_{\beta}(\mathcal{V}_{N,0})=\mathbf{P}_{\beta}(V_i=0\text{ for }i=0,\ldots,N)=(1/c_\beta)^N.
\end{equation}
Hence $g_\beta(0)=-\log c_\beta$.\\
\\
(ii) Applying again (\ref{propV1}), we observe that for all $p,q\in\mathbb{N},q>0,0\leq p\leq q$ and $\alpha_1,\alpha_2\in\mathbb{Q}^+,N\in N_{\alpha_1}\cap N_{\alpha_2}$
\begin{equation}
\mathbf{P}_{\beta}(\mathcal{V}_{qN,p\alpha_1 N+(q-p)\alpha_2 N})\geq\mathbf{P}_{\beta}(\mathcal{V}_{N,\alpha_1 N})^p\;\mathbf{P}_{\beta}(\mathcal{V}_{N,\alpha_2 N})^{q-p}.
\end{equation}
Therefore
\begin{multline}
\frac{1}{qN}\log\mathbf{P}_{\beta}\bigg(\mathcal{V}_{qN,\left(\tfrac{p}{q}\alpha_1+\big(1-\tfrac{p}{q}\big)\alpha_2\right)qN}\bigg)
\geq\tfrac{p}{qN}\,\log\mathbf{P}_{\beta}(\mathcal{V}_{N,\alpha_1 N})+\, \tfrac{q-p}{qN}\, \log\mathbf{P}_{\beta}(\mathcal{V}_{N,\alpha_2 N}),
\end{multline}
which proves that
\begin{equation}
g_\beta\Big(\tfrac{p}{q}\,\alpha_1+\big(1-\tfrac{p}{q}\big)\,\alpha_2\Big)\geq\tfrac{p}{q}\, g_\beta(\alpha_1)+\big(1-\tfrac{p}{q}\big)\,g_\beta(\alpha_2),
\end{equation}
which is the desired concavity.

Now we will show that $g_\beta(\alpha)\to 0$ as $\alpha\to\infty$ and since $g_\beta$ is concave on $\mathbb{Q}^+$, it will be sufficient to conclude that $g_\beta$ is nondecreasing.

Assume that $g_\beta(\alpha)$ does not converge to $0$ as $\alpha\to\infty$, then the concavity of $g_\beta$ insures that either $g_\beta$ is nondecreasing on $\mathbb{Q}^+$ and there exists $M>0$ such that $g_\beta(\alpha)\leq-M$ for all $\alpha\in\mathbb{Q}^+$ or $g_\beta$ is decreasing for $\alpha$ large enough and converges to $-\infty$ as $\alpha\to\infty$. In both cases we can claim that there exists $M>0$ and $\alpha_M>0$ such that $g_\beta(\alpha)\leq-M$ for all $\alpha\geq\alpha_M$. Thus, we can use \eqref{eq:supadd} to obtain
\begin{equation}\label{simp}
\mathbf{P}_{\beta}(\mathcal{V}_{N,\alpha N})\leq e^{-NM}\text{ for all }N\in\mathbb{N},\alpha\geq\alpha_M.
\end{equation}
For $\alpha\in[\alpha_M,\infty)\cap2\mathbb{N}$, we consider the set
\begin{equation}
\mathcal{N}_\alpha=\{V:V_1=3\alpha/2+1,\,\alpha+1<V_i<2\alpha+1\text{ for }i=2,\ldots N,\,V_{N+1}=0\}.
\end{equation}
For $N>\alpha$, we observe that
\begin{equation}
\mathcal{N}_\alpha\subseteq\Big\{V:V_{N+1}=0,\,\alpha(N+1)\leq\sum_{i=0}^{N+1}|V_i|\leq(2\alpha+1)(N+1)\Big\},
\end{equation}
and hence, \eqref{simp} allows us to write
\begin{equation}\label{eq:Nalpha}
\mathbf{P}_{\beta}(\mathcal{N}_\alpha)\leq\sum_{k=\alpha(N+1)}^{(2\alpha+1)(N+1)}\mathbf{P}_{\beta}(\mathcal{V}_{N+1,k})\leq(\alpha+1)(N+1)e^{-(N+1)M}.
\end{equation}
Now, we want to exhibit a lower bound on $\mathbf{P}_{\beta}(\mathcal{N}_\alpha)$. By using the Markov property, we have, with $V^*_N:=\max_{1\leq n\leq N}|V_n|$,
\begin{multline}\label{eq:fh}
\mathbf{P}_{\beta}(\mathcal{N}_\alpha)=\mathbf{P}_{\beta}\big(v_1=\tfrac{3\alpha}{2}+1\big)\sum_{k=-\alpha/2-1}^{\alpha/2-1}\mathbf{P}_{\beta}\Bigl(V^*_{N-1}<\tfrac\alpha2;V_{N}=k\Bigr)\\
\cdot\mathbf{P}_{\beta}\big(v_1=-\tfrac{3\alpha}{2}-1-k\big).
\end{multline}
Since $\mathbf{P}_{\beta}(v_1=-\tfrac{3\alpha}{2}-1-k)\geq\mathbf{P}_{\beta}(v_1=-2\alpha-1)$ for $k\in\{-\alpha/2,\ldots,\alpha/2\}$, equation \eqref{eq:fh} implies
\begin{equation}
\mathbf{P}_{\beta}(\mathcal{N}_\alpha)\geq\tfrac{e^{-\frac{\beta}{2}\left(\frac{7\alpha}{2}+2\right)}}{c_\beta^2}\sum_{k=-\alpha/2-1}^{\alpha/2-1}\mathbf{P}_{\beta}\Bigl(V^*_{N-1}<\alpha/2;V_{N}=k\Bigr).
\end{equation}
We choose $\alpha>4$ to get
\begin{equation}\label{eq:Kolmogorov}
\mathbf{P}_{\beta}(\mathcal{N}_\alpha)\geq\tfrac{e^{-2\beta\alpha}}{c_\beta^2}\mathbf{P}_{\beta}\Bigl(V^*_N<\alpha/2\Bigr),
\end{equation}
and we can apply the Kolmogorov's inequality (see \cite[p.~61]{RD05}), which gives
\begin{equation}\label{kolmo}
\mathbf{P}_{\beta}\Bigl(V^*_N<\tfrac\alpha2\Bigr)\geq1-\tfrac{4}{\alpha^2}\mathbf{Var}_{\beta}(V_N).
\end{equation}
Therefore, \eqref{eq:Nalpha}, \eqref{eq:Kolmogorov} and \eqref{kolmo} allow us to write
\begin{equation}\label{cbf}
\tfrac{e^{-2\alpha\beta}}{c_\beta^2}\left(1-\tfrac{4}{\alpha^2}\mathbf{Var}_\beta(V_N)\right)\leq(\alpha+1)(N+1)e^{-(N+1)M}.
\end{equation}
Since the above inequality is true for all $\alpha>\alpha_M$ and $N>\alpha$, we can choose $\alpha=2\sqrt{\lambda N\mathbf{Var}_\beta(v_1)}$ with $\lambda>1$ such that for $N$ large enough \eqref{cbf} becomes
\begin{equation}\label{imp}
\tfrac{1}{c_\beta^2}\left(1-\tfrac{1}{\lambda}\right)e^{-4\beta\sqrt{\lambda N\mathbf{Var}_\beta(v_1)}}\leq(2{\textstyle\sqrt{\lambda N\mathbf{Var}_\beta(v_1)}}+1)(N+1)e^{-(N+1)M}.
\end{equation}
For $N$ large, \eqref{imp} is clearly impossible and therefore $g_\beta(\alpha)$ converges to $0$ as $\alpha\to\infty$ and $g_\beta$ is nondecreasing.
\end{proof}

It remains to show that the two definitions of $g_\beta$ in \eqref{eq:glimit1} and \eqref{eq:glimit} are equivalent. To this end, we first remark that by subadditivity, the limit in \eqref{eq:glimit1} exists for all $\alpha\in [0,\infty)$. We recall \eqref{defA} and we note that, for $\alpha\in\mathbb{Q}^+$ and $N\in N_\alpha$, we have $\mathcal{V}_{N,\alpha N}\subset\{A_N\leq\alpha N,\,V_N=0\}$. Therefore
\begin{equation}
\lim_{\substack{N\in N_\alpha\\N\to\infty}}\frac{1}{N}\log\mathbf{P}_{\beta}(\mathcal{V}_{N,\alpha N})\leq \lim_{N\to\infty}\frac{1}{N}\log\mathbf{P}_{\beta}\big(A_N\leq\alpha N,V_N=0\big).
\end{equation}
We note also that $\{A_N\leq\alpha N,\,V_N=0\}=\cup_{i=0}^{\alpha N}\,\mathcal{V}_{N,i}$ and we use \eqref{eq:supadd} and the fact that $g_\beta$ is nondecreasing to state
\begin{equation}\label{inter}
\tfrac{1}{N}\log\mathbf{P}_{\beta}\left(A_N=i,V_N=0\right)\leq g_\beta\left(\tfrac{i}{N}\right)\leq g_\beta\left(\alpha\right),\quad \text{for}\;i\leq\alpha N,
\end{equation}
where $g_\beta$ in \eqref{inter} must be taken in the sense of its definition in \eqref{eq:glimit}. Thus,
\begin{equation}\label{wgt}
\mathbf{P}_{\beta}\Big(A_N\leq\alpha N,V_N=0\Big)\leq(\alpha N+1)e^{N g_\beta(\alpha)}
\end{equation}
and it suffices to take $\frac{1}{N}\log$ of both sides in \eqref{wgt} and to let $N\to\infty$ to conclude that the two definitions are indeed equivalent.

\subsection{Asymptotics of \texorpdfstring{$g_\beta$}{g}}\label{assym}
\begin{proposition}\label{gbounded}
For all $\beta>0$, there exists $c_1>0$ (depending on $\beta$) such that 
\begin{equation}\label{lb}
g_\beta(\alpha)\geq-\frac{c_1}{\alpha^2},\quad\text{for $\alpha$ large enough}.
\end{equation}
For all compact $K\subset(0,+\infty)$, there exist $c_2,\alpha_2>0$ (depending on $K$) such that  
\begin{equation}\label{ub}
g_\beta(\alpha)\leq-\frac{c_2}{\alpha^2},\quad\text{for}\;\beta\in K,\,\alpha\geq\alpha_2. 
\end{equation}
\end{proposition}
\begin{proof}
We will first prove Proposition \ref{gbounded} subject to Lemmas \ref{rwlb} and \ref{rwub} below. The proofs of these two Lemmas will be postponed to Sections \ref{pl1} and \ref{pl2}. We recall \eqref{defA} and the notation $V^*_N=\max_{1\leq n\leq N}|V_n|$.
\begin{lemma}\label{rwlb}
For $\beta>0$, there exists $c_1>0$ (depending on $\beta$) such that for $\alpha$ large enough 
\begin{equation}
\mathbf{P}_{\beta}\bigl(V_N^*\leq\alpha\bigr)\geq e^{-\frac{c_1N}{\alpha^2}},\quad\text{for $N$ large enough}.
\end{equation}
\end{lemma}
\begin{lemma}\label{rwub}
Let $K$ be a compact subset of $(0,+\infty)$. There exist $c_2,\alpha_2>0$ (depending on $K$) such that for $\beta\in K,\;\alpha\geq\alpha_2$ 
\begin{equation}
\mathbf{P}_{\beta}\bigl(A_N\leq\alpha N\bigr)\leq e^{-\frac{c_2N}{\alpha^2}},\quad\text{for $N$ large enough}.
\end{equation}
\end{lemma}

Recall \eqref{defA}, \eqref{defVV} and \eqref{eq:glimit} and note that the set $\mathcal{V}_{N,\alpha N}$ is included in $\{(V_n)_{n=0}^N\colon A_N\leq\alpha N\}$ when $\alpha\in\mathbb{Q}^+$ and $N\in N_\alpha$. Therefore, the upper bound in \eqref{ub} is a direct consequence of Lemma \ref{rwub} and of the continuity of $\alpha\to g_\beta(\alpha)$. 

For the lower bound in \eqref{lb}, by Markov property, we obtain
\begin{align}\label{yu}
\nonumber\textstyle\mathbf{P}_{\beta}(V^*_N\leq\alpha,V_N=0)&=\sum_{x\in[-\alpha,\alpha]}\mathbf{P}_{\beta}(V^*_{N-1}\leq\alpha,V_{N-1}=x)\,\mathbf{P}_{\beta}(v_1=-x)\\
&\geq\mathbf{P}_{\beta}(V^*_{N-1}\leq\alpha)\,\mathbf{P}_{\beta}(v_1=\alpha).
\end{align}
Since $V_N^*\leq\alpha$ implies $A_N=\sum_{n=1}^N|V_n|\leq\alpha N$, we can use \eqref{yu} to write
\begin{equation}\label{rt}
\textstyle\mathbf{P}_{\beta}\bigl(A_N\leq\alpha N,V_N=0\bigr)\geq\mathbf{P}_{\beta}(V^*_{N-1}\leq\alpha)\,\mathbf{P}_{\beta}(v_1=\alpha).
\end{equation}
Recall the definition of $g_\beta$ in \eqref{eq:glimit1} and apply Lemma \ref{rwlb}, we obtain $g_\beta(\alpha)\geq-\frac{c_1}{\alpha^2}$ for $\alpha$ large enough, and the proof of Proposition \ref{gbounded} is complete.
\end{proof}

\section{Proof of the main results}

In Section \ref{pth1}, we prove the variational formula stated in Theorem \ref{Thm1}. In Section \ref{sec42}, we deduce from the variational formula that the collapse transition exists and we compute the critical point (Theorem \ref{Thm2}). In section \ref{sec43}, we prove that the collapse transition is of order $3/2$ (Theorem \ref{Thm3}), although this proof is subject to Lemmas \ref{rwlb} and \ref{rwub}, that provide the asymptotic of $g_\beta(\alpha)$ as $\alpha\to\infty$. These two lemmas are established in Sections \ref{pl1} and \ref{pl2}, respectively.

\subsection{Proof of Theorem \ref{Thm1}}\label{pth1}
\begin{proof}
Recall Proposition \ref{tt}, let $N'=N+1$ and consider the partition function of size $L-1$
\begin{align}\label{dert}
\nonumber Z^\textsf{m}_{L-1,\beta}&=c_\beta\Phi^\mathsf{m}_{L-1,\beta}\sum_{N'=2}^{L}\left(\Gamma^\mathsf{m}(\beta)\right)^{N'-1}\mathbf{P}_{\beta}(\mathcal{V}_{N',L-N'})\\
&=\frac{c_\beta}{\Gamma^\mathsf{m}(\beta)}\Phi^\mathsf{m}_{L-1,\beta}\sum_{N'=2}^{L}\left(\Gamma^\mathsf{m}(\beta)\right)^{N'}\mathbf{P}_{\beta}(\mathcal{V}_{N',L-N'}).
\end{align}
(a) The lower bound:\\
\\
Pick $\alpha\in(0,1]\cap\mathbb{Q}$, $L\in N_\alpha$ and restrict the summation in \eqref{dert} to $N'=\alpha L$
\begin{equation}\label{uu}
Z^\textsf{m}_{L-1,\beta}\geq\frac{c_\beta}{\Gamma^\mathsf{m}(\beta)}\Phi^\mathsf{m}_{L-1,\beta}\left(\Gamma^\mathsf{m}(\beta)\right)^{\alpha L}\mathbf{P}_{\beta}(\mathcal{V}_{\alpha L,L-\alpha L}).
\end{equation}
Take $\frac{1}{L}\log$ of both sides in \eqref{uu} and let $L\to\infty$ to get
\begin{equation}
\tilde{f}^\mathsf{m}(\beta)\geq\alpha\log\left(\Gamma^\mathsf{m}(\beta)\right)+\alpha\,g_\beta\left(\tfrac{1-\alpha}{\alpha}\right).
\end{equation}
By continuity of $\alpha\to g_\beta(\alpha)$ on $[0,\infty)$, we can conclude that
\begin{equation}
\tilde{f}^\mathsf{m}(\beta)\geq\sup_{\alpha\in(0,1]}\left[\alpha\log\left(\Gamma^\mathsf{m}(\beta)\right)+\alpha\,g_\beta\left(\tfrac{1-\alpha}{\alpha}\right)\right].
\end{equation}
(b) The upper bound:\\
\\
For $L\in\mathbb{N}$, let
\begin{equation}
M(L)=\sup_{\alpha\in\left\{2/L,\ldots,1\right\}}\big[\left(\Gamma^\mathsf{m}(\beta)\right)^{\alpha L}\mathbf{P}_{\beta}(\mathcal{V}_{\alpha L,L-\alpha L})\big],
\end{equation}
and use \eqref{dert} to observe that
\begin{equation}\label{eq:ineq}
\tfrac{1}{L}\log{Z^\mathsf{m}_{L-1,\beta}}-\tfrac{1}{L}\log{\Phi^\mathsf{m}_{L-1,\beta}}\leq\tfrac{1}{L}\log{\big(\tfrac{c_\beta}{\Gamma^\mathsf{m}(\beta)}\big)}+\tfrac{1}{L}\log L+\tfrac{1}{L}\log{M(L)}.
\end{equation}
With the help of \eqref{eq:supadd}, we can claim that 
\begin{equation*}\label{rdr} \tfrac{1}{L}\log{M(L)}\leq\sup_{\alpha\in(0,1]}\left[\alpha\log\left(\Gamma^\mathsf{m}(\beta)\right)+\alpha\,g_\beta\left(\tfrac{1-\alpha}{\alpha}\right)\right],
\end{equation*}
which, together with \eqref{eq:ineq}, is sufficient to obtain the upper bound.
\end{proof}

\subsection{Proof of Theorem \ref{Thm2}}\label{sec42}
\begin{proof}
To begin with, we will show that $\tilde{f}^\mathsf{m}(\beta)>0$ if and only if $\Gamma^\mathsf{m}(\beta)>1$. From Theorem \ref{Thm1} and since $g_\beta$ is negative (see Lemma \ref{funcg}), it follows that if $\Gamma^\mathsf{m}(\beta)\leq1$
\begin{equation}
\tilde{f}^\mathsf{m}(\beta)=\sup_{\alpha\in[0,1]}\left[\alpha\log\left(\Gamma^\mathsf{m}(\beta)\right)+\alpha\,g_\beta\left(\tfrac{1-\alpha}{\alpha}\right)\right]\leq 0.
\end{equation}
Recall that, by Lemma \ref{le1}, $\tilde{f}^\mathsf{m}(\beta)\geq 0$ for all $\beta>0$. Thus $\tilde{f}^\mathsf{m}(\beta)=0$ when $\Gamma^\mathsf{m}(\beta)\leq 1$.\\
\\
When $\Gamma^\mathsf{nu}(\beta)>1$ in turn, Lemma \ref{funcg} gives that $g_\beta\left(\frac{1-\alpha}{\alpha}\right)$ is nondecreasing and tends to $0$ when $\alpha\to 0$. Consequently, 
\begin{equation}
\tilde{f}^\mathsf{m}(\beta)=\sup_{\alpha\in[0,1]}\left[\alpha\log\left(\Gamma^\mathsf{m}(\beta)\right)+\alpha\,g_\beta\left(\tfrac{1-\alpha}{\alpha}\right)\right]>0.
\end{equation}
By recalling the definition of $\Gamma^\mathsf{m}(\beta)$ and $c_\beta$ in \eqref{sqq} and \eqref{lawP}, we note that $\beta\mapsto\Gamma^\mathsf{m}(\beta)$ is decreasing on $[0,\infty)$ and therefore, the collapse transition occurs at $\beta^\mathsf{m}_c$, the unique positive solution of the equation $\Gamma^\mathsf{m}(\beta)=1$.
\end{proof}

\subsection{Proof of Theorem \ref{Thm3}}\label{sec43}
\begin{proof}
In this proof, we will focus on the non-uniform case. Again, adapting the proof to the uniform case is straightforward. For $\beta<\beta_c^\mathsf{m}$, let $\epsilon=\beta_c^\mathsf{nu}-\beta$. A first order Taylor expansion of $\Gamma^\mathsf{nu}(\beta)$ near $\beta_c^\mathsf{nu}$ gives
\begin{equation}
\Gamma^\mathsf{nu}(\beta)=1+c_{\mathsf{nu}}\epsilon+o(\epsilon)\ \ \text{as}\ \ \epsilon\downarrow 0\ \ \text{and where}\ \  c_{\mathsf{nu}}=\tfrac{2}{3}\Big[1+\tfrac{e^{-\beta_c^{\mathsf{nu}}/2}}{1-e^{-\beta_c^\mathsf{nu}}}\Big].
\end{equation}
Thus, we can choose $\epsilon_0\in(0,\beta_c^\mathsf{nu})$ such that for all $\beta\in I_0:=(\beta_c^\mathsf{nu}-\epsilon_0,\beta_c^\mathsf{nu})$, we have $\log{(\Gamma^\mathsf{nu}(\beta))}\leq 2c_{\mathsf{nu}}\epsilon$. By Proposition \ref{gbounded}, there exist two constants $c_2:=c_2(\beta_c^\mathsf{nu},\epsilon_0)>0$ and $\alpha_2:=\alpha_2(\beta_c^\mathsf{nu},\epsilon_0)\in(0,1)$, such that for all $\beta\in I_0$
\begin{equation}\label{eq:bb}
g_\beta\left(\tfrac{1-\alpha}{\alpha}\right)\leq-\tfrac{c_2\alpha^2}{(1-\alpha)^2}\leq-c_2\alpha^2,\quad\alpha\in[0,\alpha_2].
\end{equation}
For $\beta\in I_0$, we can write $\tilde{f}^\mathsf{nu}(\beta)=\max\{A^\mathsf{nu}_{\alpha_2,\beta},\,B^\mathsf{nu}_{\alpha_2,\beta}\}$ with
\begin{align}
\nonumber
A^\mathsf{nu}_{\alpha_2,\beta}&=\sup_{\alpha\in[0,\alpha_2)}\left[\alpha\log\left(\Gamma^\mathsf{nu}(\beta)\right)+\alpha\,g_\beta\left(\tfrac{1-\alpha}{\alpha}\right)\right]\\
B^\mathsf{nu}_{\alpha_2,\beta}&=\sup_{\alpha\in[\alpha_2,1]}\left[\alpha\log\left(\Gamma^\mathsf{nu}(\beta)\right)+\alpha\,g_\beta\left(\tfrac{1-\alpha}{\alpha}\right)\right].
\end{align}
By \eqref{eq:bb}, we can claim that for $\alpha\in[0,\alpha_2)$, we have
\begin{equation}
\alpha\log\left(\Gamma^\mathsf{nu}(\beta)\right)+\alpha\,g_\beta\left(\tfrac{1-\alpha}{\alpha}\right)\leq 2c_\mathsf{nu}\alpha\epsilon-c_2\alpha^3,
\end{equation}
and we recall that (by Proposition \ref{gbounded}) $\log\left(\Gamma^\mathsf{nu}(\beta)\right)+g_\beta\left(\tfrac{1-\alpha}{\alpha}\right)>0$ when $\alpha$ is chosen small enough.
Therefore
\begin{equation}
0<A^\mathsf{nu}_{\alpha_2,\beta}\leq\sup_{\alpha\in[0,\alpha_2)}\left[2c_\mathsf{nu}\alpha\epsilon-c_2\alpha^3\right].
\end{equation}
Since $g_\beta$ is increasing, it suffices to apply \eqref{eq:bb} at $\alpha_2$ to obtain 
\begin{equation}
g_\beta\left(\tfrac{1-\alpha}{\alpha}\right)\leq g_\beta\left(\tfrac{1-\alpha_2}{\alpha_2}\right)\leq-\tfrac{c_2\alpha_2^2}{(1-\alpha_2)^2},\quad\text{for}\ \alpha\in[\alpha_2,1]
\end{equation}
and therefore
\begin{equation}\label{ze}
B^\mathsf{m}_{\alpha_2,\beta}\leq\sup_{\alpha\in[\alpha_2,1]}\Big[\alpha\Bigl(2c_\mathsf{nu}\epsilon-\tfrac{c_2\alpha_2^2}{(1-\alpha_2)^2}\Bigr)\Big].
\end{equation}
At this stage, we can choose $\epsilon_1<\epsilon_0$ such that for all $\epsilon\in(0,\epsilon_1)$, the right hand side in \eqref{ze} is negative. Hence, for $\epsilon\in(0,\epsilon_1)$, we have
\begin{equation}\label{rhh}
\tilde{f}^\mathsf{nu}(\beta_c^\mathsf{nu}-\epsilon)=A^\mathsf{m}_{\alpha_2,\beta_c^\mathsf{nu}-\epsilon}\leq\sup_{\alpha\in[0,\alpha_2)}\left[2c_\mathsf{nu}\alpha\epsilon-c_2\alpha^3\right].
\end{equation}
In order to obtain the lower bound, we let $\epsilon\in(0,\beta_c^\mathsf{nu})$ and we can rewrite the partition function $Z_{L,\beta_c^\mathsf{nu}-\epsilon}^\mathsf{nu}$ as
\begin{equation}
Z_{L,\beta_c^\mathsf{nu}-\epsilon}^\mathsf{nu}=\left(\tfrac{e^{\beta_c^\mathsf{nu}-\epsilon}}{2}\right)^L\sum_{N=1}^{L}c_{\beta_c^\mathsf{nu}}\left(\tfrac{2c_{\beta_c^\mathsf{nu}}}{3e^{\beta_c^\mathsf{nu}-\epsilon}}\right)^N\sum_{\substack{l\in\mathcal{L}_{N,L}\\l_0=l_{N+1}=0}}\prod_{i=0}^{N}\frac{e^{-\frac{\beta_c^\mathsf{nu}-\epsilon}{2}|l_i+l_{i+1}|}}{c_{\beta_c^\mathsf{nu}}}.
\end{equation}
Since $\epsilon>0$ and $\frac{2c_{\beta_c^\mathsf{nu}}}{3e^{\beta_c^\mathsf{nu}}}=1$, we have
\begin{align}
\nonumber Z_{L,\beta_c^\mathsf{nu}-\epsilon}^\mathsf{nu}&\geq c_{\beta_c^\mathsf{nu}}\left(\tfrac{e^{\beta_c^\mathsf{nu}-\epsilon}}{2}\right)^L\sum_{N=1}^{L}e^{\epsilon N}\sum_{\substack{l\in\mathcal{L}_{N,L}\\l_0=l_{N+1}=0}}\prod_{i=0}^{N}\tfrac{e^{-\frac{\beta_c^\mathsf{nu}}{2}|l_i+l_{i+1}|}}{c_{\beta_c^\mathsf{nu}}}\\
&=c_{\beta_c^\mathsf{nu}}\left(\tfrac{e^{\beta_c^\mathsf{nu}-\epsilon}}{2}\right)^L\sum_{N=1}^{L}e^{\epsilon N}\mathbf{P}_{\beta_c^\mathsf{nu}}(\mathcal{V}_{N+1,L-N}).
\end{align}
Proceeding as in (4.1)-(4.6), we get, for all $\epsilon\in(0,\beta_c^\mathsf{nu})$,
\begin{equation}
\tilde{f}^\mathsf{nu}(\beta_c^\mathsf{nu}-\epsilon)\geq\sup_{\alpha\in[0,1]}\left[\alpha\epsilon+\alpha\,g_{\beta_c^\mathsf{nu}}\left(\tfrac{1-\alpha}{\alpha}\right)\right].
\end{equation}
By applying again Proposition \ref{gbounded}, we conclude that there exist two constants $c_1:=c_1(\beta_c^\mathsf{nu})>0$ and $\alpha_1:=\alpha_1(\beta_c^\mathsf{nu})\in(0,1)$ such that for all $\alpha\in[0,\alpha_1]$
\begin{equation}
g_{\beta_c^\mathsf{nu}}\left(\tfrac{1-\alpha}{\alpha}\right)\geq -\tfrac{c_1\alpha^2}{(1-\alpha)^2}\geq -\tfrac{c_1\alpha^2}{(1-\alpha_1)^2}.
\end{equation}
Therefore,
\begin{equation}
\sup_{\alpha\in[0,1]}\left[\alpha\epsilon+\alpha\,g_{\beta_c^\mathsf{nu}}\left(\tfrac{1-\alpha}{\alpha}\right)\right]\geq\sup_{\alpha\in[0,\alpha_1)}\left[\alpha\epsilon-\tfrac{c_1\alpha^3}{(1-\alpha_1)^2}\right],
\end{equation}
and for $\epsilon\in(0,\beta_c^\mathsf{nu})$, we have
\begin{equation}\label{rfg}
\tilde{f}^\mathsf{nu}(\beta_c^\mathsf{nu}-\epsilon)\geq\sup_{\alpha\in[0,\alpha_1)}\left[\alpha\epsilon-c'_1\alpha^3\right],
\end{equation}
where $c'_1=\tfrac{c_1}{(1-\alpha_1)^2}$. Since any function of type $\alpha\mapsto c_3\epsilon\alpha-c_4\alpha^3$ (with $c_3,c_4>0$) reaches its maximum at $\alpha=\sqrt{\frac{c_3\epsilon}{3c_4}}$, we can combine \eqref{rhh} and \eqref{rfg} and conclude that there exist $\epsilon_2>0$ and $c_5,c_6>0$ such that
\begin{equation}
c_5\epsilon^{3/2}\leq\tilde{f}^\mathsf{nu}(\beta_c^\mathsf{nu}-\epsilon)\leq c_6\epsilon^{3/2}\text{ for }\epsilon\in(0,\epsilon_2).
\end{equation}
The last estimate yields that $\tilde{f}^\mathsf{nu}(.)$ is $C^1$ but is not $C^2$ at the critical point $\beta_c^\mathsf{nu}$. This means that the non-uniform system undergoes a second order phase transition.
\end{proof}

\subsection{Proof of Lemma \ref{rwlb}}\label{pl1}
We recall the notation $V^*_N:=\max_{1\leq n\leq N}|V_n|$ and let $\mathbf{P}_{\beta,x}$ be the law of the random walk $V$ starting from $x\in\mathbb{Z}$. We also let $a=\lfloor\alpha\rfloor$ where $\lfloor\alpha\rfloor$ denotes the integer part of a real number $\alpha$. Since $V$ takes integer values only, we have
\begin{equation}
\mathbf{P}_{\beta}(V^*_N\leq\alpha)=\mathbf{P}_{\beta}(V^*_N\leq a).
\end{equation}
For $\alpha$ large, pick an integer $N$ such that $N/{a^2}\in\mathbb{N}$ and let $k=N/{a^2}$. With the help of the Markov property of $V$, we desintegrate $\mathbf{P}_{\beta}(V^*_N\leq a)$ with respect to the position occupied by the random walk $V$ at times $a^2,2a^2,\dots,(k-1)a^2$,
\begin{align}\label{tfs}
\nonumber\mathbf{P}_{\beta}(V^*_N\leq a)&=\sum_{\substack{x_0=0,x_i\in[-a,a]\\i=1,\ldots,k}}\prod_{i=0}^{k-1}\mathbf{P}_{\beta,x_i}\Bigl(V^*_{a^2}\leq a;V_{a^2}=x_{i+1}\Bigr)\\
&\geq\sum_{\substack{x_0=0,x_i\in[-a/4,a/4]\\i=1,\ldots,k}}\prod_{i=0}^{k-1}\mathbf{P}_{\beta,x_i}\Bigl(V^*_{a^2}\leq a;V_{a^2}=x_{i+1}\Bigr).
\end{align}
For any integer $x\in[0,a/4]$, we consider the two sets of paths
\begin{equation}
\Pi_1^x=\{(V_i)_{i=0}^{a^2}:V_0=x;V^*_{a^2}\leq a;V_{a^2}\in[-a/4,a/4]\},
\end{equation}
and
\begin{equation}
\Pi_2=\{(V_i)_{i=0}^{a^2}:V_0=0;V^*_{a^2}\leq3a/4;V_{a^2}\in[-a/4,0]\}.
\end{equation}
Clearly, if $V=(V_i)_{i=0}^{a^2}\in\Pi_2$, then the trajectory $V+x$ starts at $x$ and is an element of $\Pi_1^x$. Similarly, for $x\in[-a/4,0]$, $\Pi'_2+x\subseteq\Pi_1^x$ where
\begin{equation}
\Pi'_2=\{(V_i)_{i=0}^{a^2}:V_0=0;V^*_{a^2}\leq3a/4;V_{a^2}\in[0,a/4]\}.
\end{equation}
Since $V$ is symetric, the quantities $\mathbf{P}_{\beta,0}(\Pi_2)$ and $\mathbf{P}_{\beta,0}(\Pi'_2)$ are equal and therefore, for $x\in[-a/4,a/4]$,
\begin{equation}
\mathbf{P}_{\beta,x}\bigl(V^*_{a^2}\leq a;V_{a^2}\in[-a/4,a/4]\bigr)\geq\mathbf{P}_{\beta}\bigl(V^*_{a^2}\leq3a/4;V_{a^2}\in[-a/4,0]\bigr).
\end{equation}
Recall \eqref{tfs}, we conclude that
\begin{equation}\label{rh}
\mathbf{P}_{\beta}(V^*_{N}\leq a)\geq\bigl[\mathbf{P}_{\beta}\bigl(V^*_{a^2}\leq 3a/4;V_{a^2}\in[-a/4,0]\bigr)\bigr]^{k}.
\end{equation}
It remains to bound from below the right hand side of \eqref{rh}. Let $\theta_{a^2}(t)$, $t\in[0,1]$ be the continuously interpolated process associated with the RW trajectory $(V_i)_{i=0}^{a^2}$, i.e.,
\begin{equation}
\theta_{a^2}(t)=V_{\lfloor a^2t\rfloor}+\{a^2t\}v_{\lfloor a^2t\rfloor+1},\quad t\in[0,1].
\end{equation}
Let $\sigma_\beta^2=\mathbf{E}_\beta(v_1^2)$. By Donsker's theorem, $\theta_{a^2}(.)/(\sigma_\beta a)\Rightarrow B(.)$ as $a\to\infty$ on $C[0,1]$, where $B$ is a standard Brownian motion (see \cite[p.~399]{RD05}). Therefore
\begin{multline}
\lim_{a\to\infty}\mathbf{P}_{\beta}\bigl(V^*_{a^2}\leq3a/4;V_{a^2}\in[-a/4,0]\bigr)\\
=\mathbf{P}\Bigl(\max_{0\leq t\leq 1}|B(t)|\leq\tfrac{3}{4\sigma_\beta};B(1)\in\bigl[-\tfrac{1}{4\sigma_\beta},0\bigr]\Bigr)\in(0,1).
\end{multline}
Consequently, there exist $u_1\in(0,1)$ and $\alpha_0>0$ such that for all $\alpha>\alpha_0$,
\begin{equation}\label{ed}
\mathbf{P}_{\beta}(V^*_{a^2}\leq 3a/4;V_{a^2}\in[-a/4,0])\geq u_1,
\end{equation}
and \eqref{rh} becomes $\mathbf{P}_{\beta}(V^*_{N}\leq a)\geq u_1^{N/{a^2}}.$ To overcome the limitation $N/a^2\in\mathbb{N}$, we write
\begin{equation}
\mathbf{P}_{\beta}(V^*_{N}\leq a)\geq\mathbf{P}_{\beta}\big(V^*_{a^2\lfloor N/a^2\rfloor}\leq a\big)\geq u_1^{\lfloor N/a^2\rfloor}.
\end{equation}
It remains to choose $c_1>0$ satisfying $u_1>e^{-c_1}$, so that for $N$ large enough $u_1^{\lfloor N/a^2\rfloor}\geq e^{-\frac{c_1N}{\alpha^2}}$, which completes the proof.

\subsection{Proof of Lemma \ref{rwub}}\label{pl2}
The proof is a coarse graining argument divided into $3$ steps. In step 1, we pick $M\in\mathbb{N}$, we set $a=\lfloor\alpha\rfloor$ and we partition $\{1,\dots,N\}$ into $N/Ma^2$ intervals of length $Ma^2$. We show that a RW trajectory satisfying $A_N\leq\alpha N$ must, in a positive fraction of these $N/Ma^2$ intervals, spend more than a third of the time at distance less than $2\alpha$ from $0$. This gives the upper bound
\begin{equation}\label{gsd}
\mathbf{P}_{\beta}\big(A_N\leq\alpha N\big)\leq\Bigl[4e\sup_{x\in\mathbb{Z}}\phi_{\beta,x}(2\alpha,Ma^2,Ma^2/3)\Bigr]^{\frac{N}{4Ma^2}}
\end{equation}
with
\begin{equation}\label{defphi}
\textstyle\phi_{\beta,x}(t,n,m):=\mathbf{P}_{\beta,x}\Bigl(\sum_{i=1}^{n}\mathbf{1}_{\{|V_i|<t\}}\geq m\Bigr),\quad t,m\in[0,\infty),\,x\in\mathbb{Z},n\in\mathbb{N}.
\end{equation}
In step $2$, we prove that we can remove the supremum over the starting position of $V$ in \eqref{gsd} by simply enlarging $\alpha$ to $2\alpha$. To be more specific we will show that
\begin{equation}\label{unsup}
\sup_{x\in\mathbb{Z}}\phi_{\beta,x}(\alpha,N,N/3)\leq\phi_{\beta,0}(2\alpha,N,N/4),\quad\text{for $\alpha>0$ and $N\in\mathbb{N}$}.
\end{equation}
Finally, we will see in step 3 that, by choosing $M$ large enough, there exists $\alpha_M>0$ such that for all $\alpha>\alpha_M$
\begin{equation}\label{boup}
\sup_{\beta\in K}\phi_{\beta,0}(4\alpha,Ma^2,Ma^2/4)<1/4e.
\end{equation}
By putting together \eqref{defphi}, \eqref{unsup} and \eqref{boup}, we complete the proof of Lemma \ref{rwub}.

\begin{step}
\end{step}
For $\alpha>0$, recall that $a=\lfloor\alpha\rfloor$, $M\in\mathbb{N}$ and pick $N\in Ma^2\mathbb{N}:=\{Ma^2n, n\in\mathbb{N}\}$. Then, partition the time interval $\{1,\ldots,N\}$ into the $N/Ma^2$ subintervals $\{I_{j,M}\}_{j\in\{1,\dots,N/Ma^2\}}$ of length $Ma^2$ each, i.e.,
\begin{equation}
I_{j,M}=\{(j-1)Ma^2+1,\ldots,jMa^2\}\quad\text{for}\quad j=1,\ldots,\tfrac{N}{Ma^2}.
\end{equation}
Pick $V=(V_i)_{i=0}^N\in\{0\}\times\mathbb{Z}^N$ and consider, for each $j\leq N/Ma^2$, the number of steps made by $V$ on the time interval $I_{j,M}$ that satisfy $|V_i|\leq2\alpha$. Then, set
\begin{align}\label{resto}
&A_{N,\alpha}=\Bigl\{(V_i)_{i=0}^N:\sum_{i=0}^N|V_i|\leq\alpha N\Bigr\},\\
\nonumber &B_{N,M,\alpha}=\biggl\{(V_i)_{i=0}^N:\#\Bigl\{1\leq j\leq\frac{N}{Ma^2}:\sum_{i\in I_{j,M}}\mathbf{1}_{\{|V_i|<2\alpha\}}\geq\frac{Ma^2}{3}\Bigr\}\geq\frac{N}{4Ma^2}\biggr\}.
\end{align}
\begin{lemma}\label{ntb}
For all $\alpha>0$ and $M,N\in\mathbb{N}$ such that $N\in Ma^2\mathbb{N}$, the following relation holds true
\begin{equation}
A_{N,\alpha}\subseteq B_{N,M,\alpha}.
\end{equation}
\end{lemma}
\begin{proof}
Note that $V\in A_{N,\alpha}$ necessarily satisfies $\sum_{i=0}^N\mathbf{1}_{\{|V_i|\leq2\alpha\}}\geq N/2$ and this latter conditions is clearly not verified if $V\notin B_{N,M,\alpha}$.
\end{proof}
We recall \eqref{resto}, we pick $N\in 4Ma^2\mathbb{N}$ and we set $k:=\frac{N}{4Ma^2}\in\mathbb{N}$. By taking into account the indices of those intervals $I_{j,M}$ on which $\sum_{i\in I_{j,M}}\mathbf{1}_{\{|V_i|<2\alpha\}}\geq\frac{Ma^2}{3}$, we obtain the upper bound
\begin{equation}\label{eq:ii}
\mathbf{P}_{\beta}\big(B_{N,M,\alpha}\big)\leq\sum_{j_1<\ldots<j_k}\mathbf{E}_{\beta}\biggl(\prod_{s=1}^k\mathbf{1}_{\bigl\{\sum_{s\in I_{j_s,M}}\mathbf{1}_{\{|V_i|<2\alpha\}}\geq\frac{Ma^2}{3}\bigr\}}\biggr)
\end{equation}
where the sum is taken all over possible $k$-uples in $\{1,\ldots,\frac{N}{Ma^2}\}$. We recall \eqref{defphi} and we apply the Markov property $k$ times at $(j_1-1)Ma^2,\dots,(j_k-1)Ma^2$ to obtain
\begin{equation}\label{eq:im}
\mathbf{E}_{\beta}\biggl(\prod_{s=1}^k\mathbf{1}_{\bigl\{\sum_{s\in I_{j_s,M}}\mathbf{1}_{\{|V_i|<2\alpha\}}\geq\frac{Ma^2}{3}\bigr\}}\biggr)\leq\Bigl[\sup_{x\in\mathbb{Z}}\phi_{\beta,x}(2\alpha,Ma^2,Ma^2/3)\Bigr]^k.
\end{equation}
At this stage, Lemma \ref{ntb}, \eqref{eq:ii}, \eqref{eq:im} and the inequalities $\binom{n}{m}\leq n^m/m!$ and $m!\geq(m/e)^m$ allow us to write
\begin{equation}
\mathbf{P}_{\beta}(A_N\leq\alpha N)\leq\Bigl(\frac{eN}{Ma^2k}\Bigr)^k\Bigl[\sup_{x\in\mathbb{Z}}\phi_{\beta,x}(2\alpha,Ma^2,Ma^2/3)\Bigr]^k=\Bigl[4e\sup_{x\in\mathbb{Z}}\phi_{\beta,x}(2\alpha,Ma^2,Ma^2/3)\Bigr]^{\frac{N}{4Ma^2}}.
\end{equation}

\begin{step}
\end{step}
We want to prove \eqref{unsup}. First, we prove that
\begin{equation}
\sup_{x\notin(-\alpha,\alpha)}\phi_{\beta,x}(\alpha,N,N/3)\leq\sup_{y\in(-\alpha,\alpha)}\phi_{\beta,y}(\alpha,N,N/4). 
\end{equation}
Let $\tau_1=\inf\{n\in\mathbb{N}:V_n\in(-\alpha,\alpha)\}$. Pick $x\notin(-\alpha,\alpha)$ and apply the Markov property at time $\tau_1$ to obtain
\begin{align}\label{eq:pp}
\phi_{\beta,x}(\alpha,N,N/3)&=\mathbf{P}_{\beta,x}{\textstyle\bigl(\sum_{i=1}^{N}\mathbf{1}_{\{|V_i|<\alpha\}}\geq\tfrac{N}{3}}\bigr)\\
\nonumber&=\textstyle\sum_{n=1}^{N}\sum_{y=-a}^a\mathbf{P}_{\beta,x}\bigl(\tau_1=n,V_n=y\bigr)\,\mathbf{P}_{\beta,y}\Bigl(\sum_{i=1}^{N-n}\mathbf{1}_{\{|V_i|<\alpha\}}\geq\tfrac{N}{3}-1\Bigr).
\end{align}
We have, for $N\geq12$,
\begin{align}\label{eq:pl}
\nonumber\textstyle\mathbf{P}_{\beta,y}\Bigl(\sum_{i=1}^{N-n}\mathbf{1}_{\{|V_i|<\alpha\}}\geq\frac{N}{3}-1\Bigr)&\leq\textstyle\mathbf{P}_{\beta,y}\Bigl(\sum_{i=1}^{N}\mathbf{1}_{\{|V_i|<\alpha\}}\geq\frac{N}{4}\Bigr)\\
&\leq\sup_{y\in(-\alpha,\alpha)}\phi_{\beta,y}(\alpha,N,N/4),
\end{align}
and by plugging \eqref{eq:pl} into \eqref{eq:pp}, we can write that $\phi_{\beta,x}(\alpha,N,N/3)$ is smaller than\\ $\sup_{y\in(-\alpha,\alpha)}\phi_{\beta,y}(\alpha,N,N/4)$. The latter is valid for all $x\notin(-\alpha,\alpha)$, therefore
\begin{equation}\label{eq:ll}
\sup_{x\notin(-\alpha,\alpha)}\phi_{\beta,x}(\alpha,N,N/3)\leq\sup_{y\in(-\alpha,\alpha)}\phi_{\beta,y}(\alpha,N,N/4).
\end{equation}
Because of \eqref{eq:ll}, the step will be complete once we show that $\phi_{\beta,x}(\alpha,N,N/4)\leq\phi_{\beta,0}(2\alpha,N,N/4)$ for $x\in(-\alpha,\alpha)$. We recall that $P_{\beta,x}$ is the law of the random walk $(V_i)_{i\geq0}$ defined in \eqref{lawP} with $V_0=x$. Thus, if $(V_i)_{i\geq 0}$ follows the law $\mathbf{P}_\beta=\mathbf{P}_{\beta,0}$ then $(V_i+x)_{i\geq 0}$ follows the law $\mathbf{P}_{\beta,x}$. Moreover, for $|x|<\alpha$ the inequality $|V_i+x|<\alpha$ immediately entails $|V_i|<2\alpha$. Thus,
\begin{equation}
\textstyle\mathbf{P}_{\beta,x}\Bigl(\sum_{i=1}^{N}\mathbf{1}_{\{|V_i|<\alpha\}}\geq\frac{N}{4}\Bigr)\leq\mathbf{P}_{\beta,0}\Bigl(\sum_{i=1}^{N}\mathbf{1}_{\{|V_i|<2\alpha\}}\geq\frac{N}{4}\Bigr),
\end{equation}
which is exactly $\phi_{\beta,x}(\alpha,N,N/4)\leq\phi_{\beta,0}(2\alpha,N,N/4)$ and completes the step.

\begin{step}
\end{step}
In this step, we show \eqref{boup}. First observe that
\begin{equation}\label{rest}
\mathbf{P}_{\beta}\bigg(\sum_{i=1}^{Ma^2}\mathbf{1}_{\{|V_i|<4\alpha\}}\geq\tfrac{Ma^2}{4}\bigg)\leq\mathbf{P}_{\beta}\biggl(\sum_{i=Ma^2/5}^{Ma^2}\mathbf{1}_{\{|V_i|<4\alpha\}}\geq\tfrac{Ma^2}{20}\biggr).
\end{equation}
Let $\sigma_\beta^2=\mathbf{E}_\beta(v_1^2)$ and $\rho_\beta=\mathbf{E}_\beta(|v_1|^3)$. For all $x\in\mathbb{R}$ and $n\in\mathbb{N}$, by Berry-Esseen Theorem (see \cite[p.~124]{RD05}), we obtain
\begin{equation}\label{rrtf}
\textstyle\Bigl|\mathbf{P}_{\beta}\Bigl(\frac{|v_1+\ldots+v_n|}{\sigma_\beta\sqrt{n}}\leq x\Bigr)-\mathbf{P}\Bigl(|\mathcal{N}(0,1)|\leq x\Bigr)\Bigr|\leq\frac{6\rho_\beta}{\sigma_\beta^3\sqrt{n}}.
\end{equation}
From \eqref{rest}, we can apply the Markov inequality, which gives
\begin{equation}\label{esti} \mathbf{P}_{\beta}\biggl(\sum_{i=Ma^2/5}^{Ma^2}\mathbf{1}_{\{|V_i|<4\alpha\}}\geq\tfrac{Ma^2}{20}\biggr)\leq\tfrac{20}{Ma^2}\sum_{i=Ma^2/5}^{Ma^2}\mathbf{P}_{\beta}(|V_i|<4\alpha).
\end{equation}
For all $i\geq Ma^2/5$, we have
\begin{equation}\label{upbd}
\mathbf{P}_{\beta}(|V_i|<4\alpha)=\mathbf{P}_{\beta}\Bigl(\tfrac{|V_i|}{\sigma_\beta\sqrt{i}}<\tfrac{4\alpha}{\sigma_\beta\sqrt{i}}\Bigr)\leq\mathbf{P}_{\beta}\Bigl(\tfrac{|V_i|}{\sigma_\beta\sqrt{i}}\leq\tfrac{4\sqrt{5}}{\sigma_\beta\sqrt{M}}\Bigr).
\end{equation}
By using the upper bound in \eqref{rrtf}, we can rewrite \eqref{upbd} as
\begin{equation}\label{ijk}
\mathbf{P}_{\beta}(|V_i|<4\alpha)\leq\mathbf{P}\Bigl(|\mathcal{N}(0,1)|\leq\tfrac{4\sqrt{5}}{\sigma_\beta\sqrt{M}}\Bigr)+\tfrac{6\rho_\beta}{\sigma_\beta^3\sqrt{i}}.
\end{equation}
Thus, we can use \eqref{ijk} in \eqref{esti} to obtain 
\begin{equation}\label{eq:sub}
\mathbf{P}_{\beta}\biggl(\sum_{i=Ma^2/5}^{Ma^2}\mathbf{1}_{\{|V_i|<4\alpha\}}\geq\tfrac{Ma^2}{20}\biggr)\leq 20\,\mathbf{P}\Bigl(|\mathcal{N}(0,1)|\leq\tfrac{4\sqrt{5}}{\sigma_\beta\sqrt{M}}\Bigr)+\tfrac{20}{Ma^2}\sum_{i=Ma^2/5}^{Ma^2}\tfrac{6\rho_\beta}{\sigma_\beta^3\sqrt{i}}.
\end{equation}
At this stage, we replace $\sigma_\beta$ and $\rho_\beta$ by $\sigma=\inf_{\beta\in K}\sigma_\beta$ and $\rho=\sup_{\beta\in K}\rho_\beta$ in \eqref{eq:sub} so that the inequality in \eqref{eq:sub} becomes uniform in $\beta\in K$. We can choose $M$ such that $20\,\mathbf{P}(|\mathcal{N}(0,1)|\leq \frac{4\sqrt{5}}{\sigma\sqrt{M}})<1/4e$. Since
\begin{equation}\label{steco}
\tfrac{120\rho}{Ma^2\sigma^3}\sum_{i=Ma^2/5}^{Ma^2}\tfrac{1}{\sqrt{i}}\to 0\text{ as }\alpha\to\infty,
\end{equation}
there exists $\alpha_M>0$ such that for all $\alpha\geq\alpha_M$
\begin{equation}\label{detro}
20\,\mathbf{P}\Bigl(|\mathcal{N}(0,1)|\leq \tfrac{4\sqrt{5}}{\sigma\sqrt{M}}\Bigr)+\tfrac{120\rho}{Ma^2\sigma^3}\sum_{i=Ma^2/5}^{Ma^2}\tfrac{1}{\sqrt{i}}< 1/4e.
\end{equation}
It remains to recall \eqref{defphi}, and then \eqref{rest}, \eqref{eq:sub} and \eqref{detro} are sufficient to complete the step.

\end{document}